\documentclass[12pt,a4paper]{article}

\usepackage[english]{babel}
\usepackage[T1]{fontenc}
\usepackage[utf8]{inputenc}

\usepackage{bm}
\usepackage{lmodern}
\usepackage[normalem]{ulem}

\usepackage[top=2.25cm, bottom=2.75cm, left=2.75cm, right=2.75cm]{geometry}
\usepackage{latexsym}
\usepackage{pdflscape}
\usepackage{fancyhdr}
\usepackage{setspace}

\usepackage{graphicx}
\usepackage[dvipsnames]{xcolor}
\usepackage{epstopdf}
\usepackage{epsf}
\usepackage{eepic}
\usepackage{tikz}
\usetikzlibrary{shadings,shapes,arrows,calc,positioning,shadings,decorations.markings, patterns}

\usepackage[numbers]{natbib}

\usepackage{amsmath,amssymb,amsfonts,amscd}
\usepackage{enumitem}
\usepackage{caption}
\usepackage[format=hang,margin=10pt]{subcaption}
\usepackage{listings}
\usepackage{varwidth}
\usepackage{mathrsfs}

\usepackage{algorithm}
\usepackage{algpseudocode}
\usepackage[numbered,framed]{matlab-prettifier}
\usepackage{amsthm}
\usepackage[linguistics]{forest}
\forestset{
nice empty nodes/.style={
    for tree={l sep*=2},
    delay={where content={}{shape=coordinate}{}}
},
}

\usepackage{changepage}
\usepackage{xifthen}
\usepackage{todonotes}

\usepackage{hyperref}
\usepackage{cleveref}

\numberwithin{equation}{section}
\setlist[itemize,1]{label=$\bullet$}
\setlist[itemize,2]{label=$\triangleleft$}
\setlist[enumerate,1]{label=(\roman*)}
\setlist[enumerate,2]{label=(\arabic*)}

\definecolor{TUIl-orange}{RGB}{255, 121, 0}
\definecolor{TUIl-titleblue}{RGB}{0, 68, 121}
\definecolor{TUIl-textblue}{RGB}{0, 51, 88}
\definecolor{TUIl-green}{RGB}{0, 116, 122}
\definecolor{TUIl-grey}{RGB}{165, 165, 165}

\hypersetup{
    breaklinks=true,
    unicode=true,
    pdfpagelayout=OneColumn,
    bookmarksnumbered=true,
    bookmarksopen=true,
    bookmarksopenlevel=0,
    pdfborder={0 0 0},  
    colorlinks=true,
    linkcolor=Cerulean,    
    citecolor=Cerulean,
}

\algnewcommand\algorithmicinput{\textbf{Input:}}
\algnewcommand\AlgInput{\item[\algorithmicinput]}
\algnewcommand\algorithmicoutput{\textbf{Output:}}
\algnewcommand\AlgOutput{\item[\algorithmicoutput]}

\lstMakeShortInline"
\newtheoremstyle{dotless}{}{}{\itshape}{}{\bfseries}{}{ }{}
\newtheoremstyle{no-italic}{}{}{}{}{\bfseries}{}{ }{}
\theoremstyle{dotless}
\newtheorem{Theorem}{Theorem}[section]
\newtheorem{Example}[Theorem]{Example}
\newtheorem{Lemma}[Theorem]{Lemma}
\newtheorem{Definition}[Theorem]{Definition}
\newtheorem{Assumption}{Assumption}
\newtheorem{Remark}[Theorem]{Remark}
\newtheorem{Proposition}[Theorem]{Proposition}
\newtheorem{Corollary}[Theorem]{Corollary}
\newtheorem{Test Instance}[Theorem]{Test Instance}
\theoremstyle{no-italic}

\usepackage{tikz}

\title{A Unified Characterization of Nonlinear Scalarizing Functionals in Optimization}
\author{Gemayqzel Bouza  \thanks{Faculty of Mathematics and Computer Science, University of Havana, 10400  Havana, Cuba,
{\texttt{gema@matcom.uh.cu}}} \and Ernest Quintana \thanks{Institute for Mathematics, Technische Universität Ilmenau, 98693 Ilmenau, Germany,
{\texttt{ernest.quintana-aparicio@tu-ilmenau.de}}} \and Christiane Tammer \thanks{Institute of Mathematics, Martin-Luther-Universität Halle-Wittenberg, 06126 Halle, Germany, {\texttt{christiane.tammer@mathematik.uni-halle.de}}}}
\date{}

\begin{document}

%%%%%%%%%%%%%%%%
%% Title page %%
%%%%%%%%%%%%%%%%
\maketitle

%%%%%%%%%%%%%%
%% Abstract %%
%%%%%%%%%%%%%%
\begin{abstract}
Over the years, several classes of scalarization techniques in optimization have been introduced and employed in deriving separation theorems, optimality conditions and algorithms. In this paper, we study the relationships between some of those classes in the sense of inclusion. We focus on three types of scalarizing functionals defined by Hiriart-Urruty, Drummond and Svaiter, Gerstewitz. We completely determine their relationships. In particular, it is shown that the class of the functionals by Gerstewitz is minimal in this sense. Furthermore, we define a new (and larger) class of scalarizing functionals that are not necessarily convex, but rather quasidifferentiable and positively homogeneous. We show that our results are connected with some of the set relations in set optimization.
\end{abstract}

%%%%%%%%%%%%%%%%%%%%%%%%%%%%%%%%%%
%% Key words and classification %%
%%%%%%%%%%%%%%%%%%%%%%%%%%%%%%%%%%

\vspace{2ex} \noindent {\small\textbf{Mathematics subject
classifications (MSC 2010):}}  	90C26, 90C29, 90C30, 90C48 	

%%%%%%%%%%%%%
%% Content %%
%%%%%%%%%%%%%

\section{Introduction}
\label{intro}

Scalarization methods are a fundamental concept in optimization theory. Indeed, both necessary and sufficient optimality conditions and algorithms for solving vector optimization problems can be derived with them. Furthermore, scalarizing functionals play an important role in functional analysis, risk theory and mathematical finance (see Chapter 15 in the book \cite{17} and references therein). 

In the literature, several classes of scalarization methods for vector optimization problems have been defined, see for example \cite{3,6,7,10,31}. In the first part of this paper, we will provide a unified characterization of these classes by studying their relationships in the sense of inclusion.

In \cite{29,12,10} (see also \cite{17}), an axiomatic approach to scalarization in vector optimization was introduced.  These axioms are those of monotonicity and order representability (see Definition \ref{d-axioms}). It was shown that they were necessary and sufficient for  characterizing the solution sets of a vector optimization problem, namely, those of (properly, weakly)minimal. Although the axioms are stated, it is not clear how to generate all the functionals satisfying them. In a second part of this paper we propose a partial solution to this problem. Since it is well known that the class of quasidifferentiable functions is very large (see \cite{11}), we derive necessary and sufficient conditions on the quasidifferential of a positively homogeneous functional at $0,$ under which those axioms are satisfied. 

We will consider a normed space $(Y,\|\cdot\|)$ and denote its dual by $(Y^*, \|\cdot\|_*)$. The symbol $\langle \cdot, \cdot \rangle$ represents the dual pairing. Furhtermore, we use $y^*$ when referring to a generic element of the dual space. For any subset $G$ of $K^*,$ we denote by $\sigma_G$ the support functional of $G$, that is, 

\begin{equation*}
\sigma_G(y):= \sup_{y^*\in G}\{\langle y^*,y\rangle \}.
\end{equation*}

The ball in the dual space centered at $\bar{y}^*$ and with radius $\epsilon>0$ will be denoted as $\Bbb B(y^*,\epsilon),$ i.e, 
$$\Bbb B(y^*,\epsilon):=\{y^* \in Y^*: \|y^*-\bar{y}^*\|_*\leq  \epsilon\}.$$ For a set $S\subseteq Y,$  the sets $\operatorname{int}(S)$ and $\overline{S}$ are just the interior and closure of $S$ respectively. In the case of a set $D\subseteq Y^*,$ $\operatorname{int}(D)$ is the topological interior of $D$ and $\overline{D}^*$ its $w^*-$closure. For the definition and properties of the $w^*-$ topology, see the book \cite{18}. 

The rest of the paper is organized as follows: In Section \ref{s-prel}, some important concepts that will be used in the forthcoming sections are recalled. In Section \ref{s-3}, for a given convex cone in the primal space represented as the $0$-level set of a continuous functional, we study the problem of finding a representation of its dual cone. In Section \ref{s-4}, applying the results of the previous section, the relationships between the classes of scalarizing functionals by Gerstewitz, Huriart- Urruty and Drummond-Svaiter are discussed. In Section \ref{s-5}, based on the results of the previous section, we introduce a class of quasidifferentiable scalarization functionals and show that this class is indeed larger than what has been considered previously. Section \ref{s-6} states some conclusions and remarks. 

\section{Preliminaries}\label{s-prel}

In this section, we state some basic definitions and results that will be used in the paper. We start with the concept of a cone and its dual.

\begin{Definition}

Let $X$ be a topological vector space. Then, 

\begin{enumerate}
\setlength\itemsep{1em}
\item A subset $C$ of $X$ is called a cone if $x\in C \Longrightarrow tx\in C$ for all $t\geq 0.$ A cone $C$ is convex if $C+C\subseteq C.$ Furthermore, we say that $C$ is nontrivial if $C\neq\emptyset,\; C\neq \{0\}$ and $C\neq X.$

\item A cone $C$ is said to be normal, if there exists a neighborhood base $\{V_\alpha\}_{\alpha \in T}$ of $0\in X$ such that 

$$\forall\; \alpha \in T:\;  V_\alpha= \left(V_\alpha+C\right) \cap \left(V_\alpha-C\right).$$

\item For a given cone $C\subseteq Y,$ the set

$$C^*:=\{y^* \in Y^*: \langle y^*,y \rangle\geq 0 \textrm{ for all } y\in C \}$$ is called the dual cone of $C.$

\item The set 

$$C^{s*}:=\{y^*\in Y^*: \langle y^*,y \rangle> 0 \textrm{ for all } y\in C\setminus\{0\}\}$$ is called the quasinterior of $C^*.$ 

\end{enumerate}

\end{Definition} It is well known that $C^*$ is a $w^*-$ closed and convex cone, see \cite{18}. Another important concept that we will use is that of generators of convex cones.

\begin{Definition}\label{d-generator}

Let $X$ be a topological vector space and consider a convex cone $C\subseteq X.$ We say that a set $G\subseteq X$ is a generator of $C$ if it satisfies the following properties:

\begin{enumerate}
\setlength\itemsep{1em}
\item $G$ is convex,

\item $0\notin \bar{G},$ 

\item $C=\operatorname{cone}(G).$

\end{enumerate} Here,
$$\operatorname{cone}(G):=\{tx|\; t\geq 0,\; x\in G\}$$ is read as the cone generated by $G.$  Furthermore, if $G$ is a generator of $C$ such that for every $x \in C\setminus\{0\}$ the representation 

$$x= \lambda g,\; \lambda >0,\; g\in G$$ is unique, we say that $G$ is a (topological) base of $C.$

\end{Definition}

Most of our results will heavily rely on the classical separation theorems in a locally convex topological vector space, see \cite{9}. We state them for completeness.

\begin{Theorem}(Jahn, \cite{9})\label{t - separationtheorem}
Let X be a topological vector space and $A, B\subseteq X$ be closed and convex. The following statements holds:

\begin{enumerate}
\setlength\itemsep{1em}

\item If $\operatorname{int}(A)\neq \emptyset,$ then $\operatorname{int}(A)\cap B=\emptyset$ if and only if we can find $x^*\in X^*\setminus\{0\}$ such that 

$$\sup\{\langle x^*, a\rangle : a\in A \}\leq  \inf\{\langle x^*, b\rangle : b\in B\}.$$

\item If $X$ is locally convex and $A$ is compact, then $A\cap B=\emptyset$ if and only if we can find $x^*\in X^*\setminus\{0\}$ such that 

$$\sup\{\langle x^*, a\rangle : a\in A \}< \inf\{\langle x^*, b\rangle : b\in B\}.$$

\end{enumerate}   

\end{Theorem} We will mainly consider Definition \ref{d-generator} and Theorem \ref{t - separationtheorem} in the context in which a normed space $Y$ with a convex cone $K$ are given. In that case, we will put $X=Y^*$ with the $w^*-$ topology and $C=K^*$. Furthermore, we will assume that $G$ is $w^*-$ compact.\\

For the rest of the section, we assume that $Y$ is a normed space and $K\subseteq Y$ is a convex cone. Now, consider an arbitrary functional $\Psi: Y \to \mathbb{R}.$ Several important properties related to this functional are involved for defining a robust representation of $C\subseteq Y$ in the next definition. 

\begin{Definition}\label{d-rob}
Let $\Psi: Y \to \mathbb{R}$ and $C\subseteq Y$ be such that $\operatorname{int}(C)\neq \emptyset.$ We say that $\Psi$ gives a robust representation of the set $C$ if the following properties hold:
\begin{enumerate}
\setlength\itemsep{1em}
\renewcommand{\labelitemi}{\scriptsize$\blacksquare$} 
\item $\Psi$ is continuous,
\item $C=\{y\in Y: \Psi(y)\leq 0\},$
\item $\operatorname{int}(C)=\{y\in Y: \Psi(y)< 0\}.$
\end{enumerate}
\end{Definition} 

The directional derivative of $\Psi$ at $\bar{y}$ in the direction $v\in Y,$ denoted as $\Psi'(\bar{y},v),$ is defined as

$$\Psi'(\bar{y},v):= \lim_{t\downarrow 0} \frac{\Psi(\bar{y}+tv )-\Psi(\bar{y})}{t},$$ whenever this limit exists. The functional $\Psi$ is called directionally differentiable at $\bar{y} $ if the above limit exists in every direction $v \in Y.$ Furthermore, $\Psi$ is called G\^ateaux differentiable at $\bar{y}$ if $\Psi'(\bar{y},\cdot)\in Y^*.$ 

If $\Psi$ is directionally differentiable at $\bar{y}$, the set 

$$\partial_{DH}\Psi(0):=\{y^* \in Y^*: \langle y^*,y \rangle\leq \Psi'(\bar{y},y)\textrm{ for all } y\in Y \}$$ is called the Dini-Hadamard subdifferential of $\Psi$ at $\bar{y}.$

We will mostly deal with convex functions. If $\Psi: Y\to \Bbb R$ is a convex functional and $\bar{y}\in Y$, the set 

$$\partial \Psi(\bar{y}):=\{y^* \in Y^*: \langle y^*,y-\bar{y} \rangle\leq \Psi(y)-\Psi(\bar{y}) \textrm{ for all } y\in Y\}$$ is called the (Fenchel) subdifferential of $\Psi$ at $\bar{y}.$

\begin{Definition}
Let $\Psi: Y \to \mathbb{R}.$ We say that $\Psi$ satisfies the Slater condition if there exists $\bar{y}\in Y$ such that $\Psi(\bar{y})<0.$

\end{Definition} 

\begin{Remark}\label{slaterequalrobustness}
It can be shown that, if $\Psi$ is a continuous convex functional, then Slater's condition is equivalent to the robustness of $\Psi$ in the sense of Definition \ref{d-rob}. Indeed, that robustness implies Slater's condition is trivial. For proving the converse, assume that $\Psi$ satisfies Slater's condition and let $\bar{y}$ be such that $\Psi(\bar{y})<0.$  Take any $y\in \operatorname{int}(C).$ Then, there exists $\lambda >1$ such that $y_\lambda= \bar{y}+\lambda(y- \bar{y})\in \operatorname{int}(C).$ Hence, we find $\Psi(y)=\Psi\left(\frac{1}{\lambda}y_\lambda+\left(1-\frac{1}{\lambda}\right)\bar{y}\right)\leq \frac{1}{\lambda} \Psi(y_\lambda)+ \left(1-\frac{1}{\lambda}\right) \Psi(\bar{y})<0.$  
\end{Remark}

Sublinear functionals are an important subclass of convex functions. Recall that $\Psi:Y\to \mathbb{R}$ is said to be sublinear if it is
\begin{itemize}
\setlength\itemsep{1em}
\renewcommand{\labelitemi}{\scriptsize$\blacksquare$} 

\item positively homogeneous, i.e, $\Psi(t y)=t\Psi(y),$ for all $y\in Y, t\geq 0,$
\item subadditive, i.e,  $ \Psi(y_1+y_2)\leq \Psi(y_1) +\Psi(y_2),$ for all $y_1,y_2\in Y.$
\end{itemize} 

The following proposition summarizes some useful facts about convex functions and its subdifferentials.

\begin{Proposition}[Schirotzek, \cite{13}]\label{properties of convex}
Let $\Theta: Y \to \Bbb R$ be a convex functional, and let $\bar{y}\in Y.$ Assume that $\Theta$ is continuous at $\bar{y}.$ The following properties hold:

\begin{enumerate}
\setlength\itemsep{1em}
\renewcommand{\labelitemi}{\scriptsize$\blacksquare$} 
\item $\Theta$ is locally Lipshitz at each point $\bar{y}\in Y,$ i.e, there exists a neighborhood $\mathcal{U}$ of $\bar{y}$ and a constant $L >0$ such that
$$\forall\; x,y\in \mathcal{U}: \;|\Psi(x)-\Psi(y)|\leq L\|x-y\|.$$  
\item The subdifferential of $\Theta$ at $\bar{y}$ is nonempty, i.e, $\partial \Theta (\bar{y})\neq \emptyset.$ Furthermore, $\partial \Theta (\bar{y})$ is convex and $w^*-$compact,
\item Consider $\Psi: Y\to \Bbb R$ as $\Psi(y)=\Theta'(\bar{y},y).$ Then, $\Psi$ is sublinear and $\Psi = \sigma_{\partial \Theta (\bar{y})},$
\item $\partial \Psi(0)= \partial \Theta (\bar{y}).$
\end{enumerate} 
\end{Proposition}

Next, we present an important type of directionally differentiable functions.

\begin{Definition}
The functional $\Psi:Y\to \Bbb R$ is said to be quasidifferentiable at $\bar{y} \in Y$ if it is directionally differentiable at $\bar{y}$ and there exists $w^*-$compact sets $G$ and $H$ such that 

$$\Psi'(\bar{y},\cdot)= \sigma_G -\sigma_H.$$ In this case, the sets $G$ and $-H$ are called the subdifferential and superdifferential of $\Psi$ at $\bar{y}$ respectively. The pair of sets $[G,-H]$ is called the quasidifferential of $\Psi$ at $\bar{y}.$ 
\end{Definition} This class of functions is very important since it includes the class of D.C.-functions (difference of convex functions) and the class of locally convex functions (with directional derivative being sublinear). Furthermore, if $\Psi$ is quasidifferentiable at $\bar{y}\in Y$ and $\partial_{DH}\Psi(\bar{y})$ is nonempty, it can be shown that 

\begin{equation}\label{eq: DH- subdif}
\partial_{DH}\Psi(\bar{y})= G \ominus H,
\end{equation} where $G \ominus H=\{ y^* \in Y^*\;| \; y^* + H\subseteq G\},$ is the so called Hadwiger-Pontryagin difference  of sets, see \cite{25,26}. For a comprehensive review on quasidifferentiability, see also \cite{11}.

A specific subdifferential concept for locally Lipschitz maps is defined via the directional derivative of Michel-Penot. Formally, given $\Psi:Y\to \mathbb{R}$ being locally Lipschitz at $\bar{y}\in Y,$ the Michel-Penot directional derivative of $\Psi$ at $\bar{y}$ in the direction $v\in Y$ is defined as 

$$\Psi^\Diamond(\bar{y},v):=\sup_{z\in Y} \limsup_{t\downarrow 0} \frac{\Psi(\bar{y}+tv +tz)-\Psi(\bar{y}+tz)}{t}.$$ In this case, the Michel-Penot subdifferential of $\Psi$ at $\bar{y}$ is  

$$\partial_{MP}\Psi (\bar{y}):=\{y^* \in Y^*:\langle y^*,y \rangle\leq  \Psi^\Diamond(\bar{y},y) \textrm{ for all } y\in Y\}.$$ Furthermore, we say that $\Psi$ is MP-regular at $\bar{y}$ if $\Psi^\Diamond (\bar{y},\cdot)= \Psi'(\bar{y},\cdot).$ Note that the class of $MP-$ regular functions at $\bar{y}$ includes in particular the class of convex continuous functions at $\bar{y}.$

The following proposition reviews some well known properties of the Michel-Penot derivative and its relation with the directional derivative.
\begin{Proposition}[Schirotzek, \cite{13}]\label{mp properties}
Let $\Psi:Y\to \mathbb{R}$ be locally Lipschitz at $\bar{y}\in Y$ with constant $L>0.$ Then,
\begin{itemize}
\setlength\itemsep{1em}
\renewcommand{\labelitemi}{\scriptsize$\blacksquare$} 
\item[(i)] $\Psi^\Diamond(\bar{y},\cdot)$ is well defined and sublinear. Furthermore, if $\Psi'(\bar{y},v)$ exists, the inequality 

$$\Psi'(\bar{y},v)\leq \Psi^\Diamond(\bar{y},v)\leq L\|v\|$$ holds,
\item[(ii)] $\partial_{MP}\Psi (\bar{y})$ is a nonempty, convex and $w^*-$compact subset of $Y^*$ such that 

$$\Psi^\Diamond(\bar{y},\cdot)=\sigma_{\partial_{MP}\Psi (\bar{y})},$$

\item[(iii)] If $\Psi$ is directionally differentiable at $\bar{y},$ then $\partial_{DH}\Psi(\bar{y})\subseteq \partial_{MP}\Psi (\bar{y}).$
\end{itemize}

\end{Proposition}

\section{Generators of dual cones}\label{s-3}

Let $\Psi: Y \to \Bbb R$ be a given continuous functional and assume that the set 

$$C:=\{y\in Y: \Psi(y)\leq 0\}$$ is a convex cone. In this section, we are concerned with the problem of finding a suitable representation for $C^*,$ the dual cone of $C,$ in terms of a subdifferential of $\Psi$. This problem has already been studied in the literature: see \cite{21} for a treatment of the case in which $\Psi$ is convex, and \cite{22} and the references therein for results when $\Psi$ is quasiconvex and locally Lipschitz. In these results, a constraint qualification (of Slater's type) is assumed. Furthermore, in \cite{23}, this constraint qualification is removed, and, assuming convexity of $\Psi,$ a representation of the dual cone was obtained as the Painleve-Kuratowski upper limit of cones generated by the subdifferential operator. In this section, we analyze the case in which $\Psi$ is quasidifferentiable and we obtain an approximate representation of $C^*$ in terms of the Dini - Hadamard and Michel - Penot subdifferentials. We start with a simple lemma.   

\begin{Lemma}\label{wcomp}

Let $G$ be a $w^*-$compact subset of $Y^*$ such that $0\notin G$ and $J$ a closed subset of $\mathbb{R}.$ Set $$D:= \{ty^*:t\in J, \;y^* \in G\}.$$ Then, $D$ is $w^*-$closed.  
\end{Lemma}

\begin{proof}
Let $\{y^*_\alpha\}_I$ be a convergent net in $D.$ Then, there are nets 

$$\{z_\alpha^*\}_I\subseteq G,\;\{t_\alpha\}_I\subseteq J,$$ and an element $\bar{y}^* \in Y^*$ such that, for each $\alpha \in I,$ 

%$$y^*_\alpha = t_\alpha z_\alpha^*, y^*_\alpha \overset{w^*}{\longrightarrow} \bar{y}^*. $$

$$y^*_\alpha = t_\alpha z_\alpha^*,\;\;\; y^*_\alpha \overset{w^*}{\longrightarrow} \bar{y}^*.  $$
Since, by assumption, $G$ is $w^*-$compact, w.l.o.g we can assume that there exists $\bar{z}^*\in G$ such that $z_\alpha^* \overset{w^*}{\longrightarrow} \bar{z}^*.$  Since $0 \notin G,$ it follows that $\bar{z}^* \neq 0.$ Hence, we can find $\bar{y} \in Y$ such that $\langle\bar{z}^*,\bar{y}\rangle=1.$ Now, by the $w^*-$ convergence of $\{y_\alpha^*\}_I$, we have $t_\alpha \langle z_\alpha^*,\bar{y}\rangle\longrightarrow   \langle\bar{y}^*,\bar{y}\rangle,$ from which we deduce that $t_\alpha \longrightarrow \langle \bar{y}^*,\bar{y}\rangle.$ Since $J$ is closed, we have $\bar{t}:=\langle \bar{y}^*,\bar{y}\rangle\in J.$ Hence, we get

$$\bar{y}^*= \lim_{\alpha \in I}t_\alpha z_\alpha^* = \bar{t} \bar{z}^* \in D,$$ as desired. 
\end{proof}

Next, we proceed to analyze the case on which $\Psi$ is sublinear. 
  
\begin{Lemma}\label{dualconesublinear}
Given a continuous sublinear functional $\Psi:Y\to \mathbb{R},$ consider the convex cone 

$$C:=\{y\in Y: \Psi(y)\leq 0\}.$$ Then, we have

$$C^*=  \overline{\operatorname{cone}}^{*}\left(-\partial \Psi(0) \right).$$ Furthermore, if $\Psi$ satisfies Slater's condition,  

$$C^*= \operatorname{cone}\left(-\partial \Psi(0) \right).$$
\end{Lemma} 
\begin{proof}

 Let 

$$D_{\Psi}:= \overline{\operatorname{cone}}^{*}\left(-\partial \Psi(0) \right).$$  Take any $y^* \in \partial \Psi(0).$ We have then 

$$\forall \;y\in Y:\;\Psi(y)\geq \Psi(0)+\langle y^*,y\rangle=\langle y^*,y\rangle.$$ This implies that $$C\subseteq C_{y^*}:=\{y\in Y: \langle y^*,y\rangle\leq 0\}$$ and, hence, $$C_{y^*}^*\subseteq C^*.$$ From this we deduce $$\bigcup_{y^* \in \partial\Psi(0)}C_{y^*}^*\subseteq C^*, $$ and since $C^*$ is $w^*-$closed, we also have $$\overline{\bigcup_{y^* \in \partial\Psi(0)}C_{y^*}^*}^*\subseteq C^*.$$  By Lemma 9.6.1 in \cite{21}, we have $C_{y^*}^*= \operatorname{cone}(-\{y^*\}),$ so that  

$$\overline{\bigcup_{y^* \in \partial\Psi(0)}C_{y^*}^*}^*= \overline{\bigcup_{y^* \in \partial\Psi(0)}}^* \operatorname{cone}(-y^*)= \overline{\operatorname{cone}}^{*}\left(\bigcup_{y^* \in \partial \Psi(0)}-\{y^*\} \right)=\overline{\operatorname{cone}}^{*}\left(-\partial \Psi(0) \right)=D_{\Psi}.$$ This means that

$$D_{\Psi} \subseteq C^*.$$ Now let us assume that $$\bar{y}^* \notin D_{\Psi}.$$ By Theorem \ref{t - separationtheorem} (ii), there is a $\bar{y}\in Y$ such that 
$$\langle \bar{y}^*,\bar{y}\rangle< \inf_{y^* \in D_{\Psi}} \{\langle y^*,\bar{y}\rangle\}=0.$$ We now claim that $\bar{y} \in C.$ Otherwise, it holds that $\Psi(\bar{y})> 0.$ Then we get 

$$0< \Psi(\bar{y})= \lim_{t\to 0^+} \frac{\Psi(t \bar{y})-\Psi(0)}{t}=\Psi^{'}(0 ;\bar{y})= \max\{\langle y^*,\bar{y}\rangle:y^* \in \partial \Psi(0)\}.$$ In particular, there is a $\hat{y}^* \in \partial \Psi(0)$ such that $\langle\hat{y}^*,\bar{y}\rangle>0.$ From this we arrive at 

$$\forall \;t\geq 0:\;-t\hat{y}^* \in D_{\Psi}$$ and then we conclude that

$$ 0=\inf_{y^* \in D_{\Psi}} \{\langle y^*,\bar{y}\rangle\}\leq \inf_{t\geq 0}\{-t\langle\hat{y}^*,\bar{y}\rangle\}=-\infty ,$$ a contradiction. Hence we have proved that $\bar{y}\in C$ and $\langle \bar{y}^*,\bar{y}\rangle<0,$ which means that $\bar{y}^* \notin C^*.$ This proves the first part of the Theorem.

Now assume that $\Psi$ satisfies Slater's condition, so that we have a $\bar{y}\in Y$ such that $\Psi(\bar{y})<0.$  It is easy to see that this is equivalent to the condition

$$0\notin \partial \Psi(0).$$ Since $\partial \Psi(0)$ is $w^*$-compact, applying Lemma \ref{wcomp} with $J=\Bbb R_+,$ we get that $\operatorname{cone}(-\partial \Psi(0))$ is $w^*-$closed and, hence, the second statement follows. This concludes the proof. 

\end{proof}

\begin{Remark}
Note that the first part of Lemma \ref{dualconesublinear} cannot be obtained from Proposition 9.6.1 in \cite{21}, where the convex case is analyzed. This is because of the fact that the Slater property is assumed. Hence, in our proof we exploited the sublinearity of $\Psi.$ Our next example shows that this representation without assuming Slater's condition is an intrinsic property of sublinear functionals, and that convexity is not enough. 
\end{Remark}

\begin{Example}
Consider $\Psi: \Bbb R \to \Bbb R $ as 
$$
\Psi(y) = 
     \begin{cases}
       0 &\quad\text{if } y\leq 0,\\
       y^2 &\quad\text{if } y>0.\\
     \end{cases}
$$ Then, $\Psi$ is convex and continuous. We have $$-\Bbb R_+=\{y\in \Bbb R: \Psi(y)\leq 0\}.$$ Furthermore, $\Psi$ do not satisfies Slater's condition. Then, it is easy to check that $\partial \Psi(0)=\{0\}$ and hence we have 

$$ \overline{\operatorname{cone}}(-\partial \Psi(0))= \{0\}\neq -\Bbb R_+.$$
\end{Example}
The following corollary shows that our result generalizes Theorem 2.13 and Corollary 2.14 derived by Jahn in \cite{2} for Bishop-Phelps cones.
\begin{Corollary}\label{bpcone}

Let $\bar{y}^*\in Y^*$ and consider the Bishop-Phelps cone

$$C(\bar{y}^*):=\{y\in Y: \langle \bar{y}^*, y\rangle \geq \|y\|\}.$$ Then, $$C(\bar{y}^*)^*= \overline{\operatorname{cone}}^{*}(\Bbb B(\bar{y}^*,1)).$$ If $\|\bar{y}^*\|_*>1,$ then  

$$C(\bar{y}^*)^*= \operatorname{cone}(\Bbb B(\bar{y}^*,1)).$$

\end{Corollary}
\begin{proof}
 Define $\Psi(y)=\|y\|-\langle \bar{y}^*,y \rangle.$ Obviously, $\Psi$ is sublinear and continuous. Furthermore, it is well known that $\partial \|\cdot\|(0)= \Bbb B(0,1).$ We  then find that
$$\partial\Psi(0)=\Bbb B(0,1)-\bar{y}^*,$$ and then  

$$-\partial\Psi(0)=\Bbb B(\bar{y}^*,1).$$ Applying now Lemma \ref{dualconesublinear}, the first part is obtained. 

Now assume that $\|\bar{y}^*\|_*>1.$ Then, there exists $\bar{y}$ such that $\langle \bar{y}^*, \bar{y} \rangle> \|\bar{y}\|.$ This is obviously equivalent to $\Psi(\bar{y})<0,$ proving that $\Psi$ satisfies Slater's condition. The second part now follows from the second part of Lemma \ref{dualconesublinear}. 
\end{proof}
\begin{Remark}
We want to mention that, in order to drop the closure operator in Corollary 2.14 on \cite{2}, it was assumed in addition that $Y$ was either separable or reflexive. These assumptions were not needed in Corollary \ref{bpcone} and, hence, our result is stronger.
\end{Remark}

Now we turn into the case in which $\Psi$ is quasidifferentiable. The next lemma will be needed.

\begin{Lemma}\label{nonsublinearnestedcone}
Let $\Psi: Y\to \mathbb{R}$ be continuous and directionally differentiable at 0 and let $C$ be a nontrivial cone in $Y.$ Assume that $\Psi$ gives a robust representation of $C$ and, hence,

$$C=\{y\in Y: \Psi(y)\leq 0\}.$$ Then,

$$\{y\in Y: \Psi'(0,y) < 0\}\subseteq \operatorname{int}(C)\subseteq \{y \in Y: \Psi'(0,y)\leq 0\}.$$
\end{Lemma}
\begin{proof}
First, note that $\Psi(0)=0$ necessarily because of the robustness assumption. Take $y\in Y$ such that $\Psi'(0,y) < 0.$ If $y \notin \operatorname{int}(C),$ because $\Psi$ is a gives a robust representation of $C,$ we get $\Psi(y)\geq 0.$ Since $\operatorname{int}(C)\cup\{0\}$ is a cone, $\Psi(\alpha y)\geq 0$ for all $\alpha >0.$ This implies 

$$\Psi'(0,y)= \lim_{\alpha \downarrow 0} \frac{\Psi(\alpha y)- \Psi(0)}{\alpha}\geq 0,$$ a contradiction. This means that $y \in \operatorname{int}(C).$\\

Now let $y \in \operatorname{int}(C).$ By robustness, $\Psi(y)<0,$ so $\Psi(\alpha y)<0$ for all $\alpha >0.$ Now it is easy to see that for any $\alpha >0,$ the quotient

$$\frac{\Psi(\alpha y)- \Psi(0)}{\alpha}<0.$$ Taking the limit when $\alpha \to 0,$ we have $\Psi'(0,y)\leq 0.$ 
\end{proof}

We can now state the main result of this section.

\begin{Theorem}\label{dualconelipschitz}
Let $\Psi:Y\to \mathbb{R}$ be locally Lipschitz and quasidifferentiable at $0$. Assume that  

$$C:=\{y\in Y: \Psi(y)\leq 0\}$$ is a cone for which $\Psi$ gives a robust representation. Furthermore, assume that the constraint qualification $0\notin \partial_{MP}\Psi(0)$ holds. Then, we have

$$\operatorname{cone}\left(-\partial_{DH} \Psi(0) \right)\subseteq C^*\subseteq  \operatorname{cone}\left(-\partial_{MP} \Psi(0) \right).$$ If in addition $\Psi$ is MP-regular at $0,$ then 

$$C^*=  \operatorname{cone}\left(-\partial_{MP} \Psi(0) \right).$$
\end{Theorem}
\begin{proof}
Since $\Psi$ is quasidifferentiable at $0,$ we have the existence of $w^*-$compact sets $G, H \subseteq Y^*$ such that 

$$\Psi'(0,\cdot)= \sigma_{G}(\cdot)-\sigma_{H}(\cdot).$$ The inclusion $G \ominus H  +H \subseteq G$ implies in particular that 

$$\sigma_{G \ominus H}(\cdot) +\sigma_H(\cdot) \leq \sigma_G(\cdot), $$ which together with Proposition \ref{mp properties} (i), gives us
\begin{equation}\label{eq:dclbound}
\sigma_{G \ominus H}(\cdot) \leq  \sigma_G(\cdot) -\sigma_H(\cdot)= \Psi'(0,\cdot)\leq \Psi^\Diamond(0,\cdot).
\end{equation}

Note that the local Lipschitz assumption at $0$ already gives us the continuity of both $\sigma_{G \ominus H}$ and $\Psi^\Diamond(0,\cdot).$ Take now any $y\in \operatorname{int}(C).$  By \eqref{eq:dclbound} and Lemma \ref{nonsublinearnestedcone} it follows that 

$$\sigma_{G \ominus H}(y)\leq \Psi'(0,y)\leq 0,$$ or 

$$\operatorname{int}(C)\subseteq \{y \in Y: \sigma_{G \ominus H}(y)\leq 0 \}.$$ Since the last set is closed, we actually have 
$$C\subseteq \{y \in Y: \sigma_{G \ominus H}(y)\leq 0 \}.$$ On the other hand, take any $y\in Y$ such that $\Psi^\Diamond(0,y)<0.$ Again, from \eqref{eq:dclbound} and Lemma \ref{nonsublinearnestedcone} we get $\Psi'(0,y)<0$ and $y\in C.$ In fact, we have proved that 

\begin{equation}\label{nestedconesgeneral}
\{y\in Y: \Psi^\Diamond(0,y)<0\}\subseteq C\subseteq\{y \in Y: \sigma_{G \ominus H}(y)\leq 0 \}. 
\end{equation}
The constraint qualification $0\notin \partial_{MP}\Psi(0)$ implies that the first set in \eqref{nestedconesgeneral} is nonempty and open. Furthermore, together with \eqref{eq:dclbound} it also implies that both $\Psi^\Diamond(0,\cdot)$ and $\sigma_{G \ominus H}$ satisfy Slater's condition. Also, note that the directional derivative function of $\Psi'(0,\cdot)$ at $0$ is itself because it is the difference of sublinear functionals. From \eqref{eq: DH- subdif}, we know that $\partial_{DH}\Psi'(0,\cdot)(0)= G \ominus H.$ Hence, we find that 

\begin{eqnarray}\label{eq:DHsubdifgeometric}
\partial_{DH}\Psi(0)&= &\{y^* \in Y^*: \langle y^*,y\rangle \leq \Psi'(0,y) \; \forall\; y\in Y\} \nonumber\\
       &= &\{y^* \in Y^*: \langle y^*,y\rangle \leq (\Psi'(0,\cdot))'(0,y) \; \forall\; y\in Y\}\nonumber\\
       &= &\partial_{DH}\Psi'(0,\cdot)(0) \nonumber\\
       &= & G \ominus H.
\end{eqnarray}

Finally, applying Lemma \ref{dualconesublinear} and \eqref{eq:DHsubdifgeometric} we get 

$$\operatorname{cone}\left(-\partial_{DH} \Psi(0) \right)\subseteq C^*\subseteq  \operatorname{cone}\left(-\partial_{MP} \Psi(0) \right),$$ which proves the first part of the Theorem. If in addition $\Psi$ is MP-regular at $0$,then, by definition, we have $\partial_{MP} \Psi(0)= \partial_{DH} \Psi(0)$ and the equality holds. This concludes the proof.   

\end{proof}

In the next sections we will be working with scalarizing functionals that fulfill the conditions of Theorem \ref{dualconelipschitz}. In particular, the robust representation  will be a consequence of the monotonicity and representability axioms, see Definition 7.

\section{Relationships among three types of nonlinear scalarizing functionals}\label{s-4}

Throughout this section, we consider the following assumption:
\begin{Assumption}
Let $(Y,\|\cdot\|)$ be a Banach space and $K\subseteq Y$ a closed, convex and pointed cone with nonempty interior.
\end{Assumption}

We will consider different classes of scalarizing functionals that have been previously studied in the literature and show relationships between them in the sense of inclusion. As mentioned in the introduction, in \cite{10} scalarizing functionals were introduced in an axiomatic way, and it was shown that these axioms are indeed necessary and sufficient in order to characterize the sets of minimal and weakly minimal points to a vector optimization problem. Hence, the definition of monotonicity properties to scalarizing functionals is our starting point in this section. 

\begin{Definition}\label{d-axioms}

Let $\Psi:Y\to \Bbb R$ be a given functional. We say that $\Psi$ is 

\begin{enumerate}
\setlength\itemsep{1em}
\item $K-$ monotone, if $z-y\in K \Longrightarrow \Psi(y)\leq \Psi(z),$
\item strictly $K-$ monotone, if $z-y \in \operatorname{int}(K) \Longrightarrow \Psi(y)< \Psi(z),$
\item strongly $K-$ monotone, if $z-y \in K,\; y\neq z \Longrightarrow \Psi(y)< \Psi(z).$
\end{enumerate}

\end{Definition}

\begin{Definition}\label{d-axiom}
Let $\Psi:Y\to \Bbb R$ be a given continuous functional. 

\begin{enumerate}
\setlength\itemsep{1em}
\renewcommand{\labelitemi}{\scriptsize$\blacksquare$} 
\item We say that $\Psi$ satisfy the monotonicity property if $\Psi$ is strictly $K-$monotone (and hence $K-$monotone).

\item We say that $\Psi$ satisfy the representability property if\\

\begin{enumerate}
\setlength\itemsep{1em}
\item $\{y\in Y: \Psi(y)\leq 0\}\subseteq -K,$ and
\item $\{y\in Y: \Psi(y)< 0\}\subseteq - \operatorname{int}(K).$
\end{enumerate}

\end{enumerate}
\end{Definition}

Although several nonlinear scalarization techniques have been defined in order to solve vector optimization problems, in this paper we deal with some of the most prominent general purpose examples. We formally define these classes:

\begin{Definition}\label{d-func}

Suppose that Assumption 1 is fulfilled. The following classes of functionals are introduced:

\begin{itemize}
\setlength\itemsep{1em}
\renewcommand{\labelitemi}{\scriptsize$\blacksquare$} 
\item Separating functionals with uniform level sets (see Gerstewitz \cite{6}, Gerth, Weidner \cite{16} and Gerstewitz, Iwanow \cite{geiw85})\\

Take any $r \in \operatorname{int}(K).$ Then, the associated functional is 

\begin{equation}\label{funcGW}
\Psi_r(y):= \min\{t\in \mathbb{R}: tr\in y+K\}.
\end{equation}
The class of all functionals of the type \eqref{funcGW} when $r$ varies over $\operatorname{int}(K)$ will be denoted as $\Omega_{GW}.$

\item Hiriart-Urruty functional (see \cite{3})\\

These functionals need as a parameter an equivalent norm $\|\cdot\|'$ to $\|\cdot\|$ and, for such a norm, it is defined as

\begin{equation}\label{funcHU}
\Psi_{\|\cdot\|'}(y):= d(y,-K)- d(y, Y\setminus -K),
\end{equation}
where $d(y, A)$ denotes the usual distance from $y$ to the set $A$ with respect to $\|\cdot\|'.$ The class of all of such functionals when $\|\cdot\|'$ varies over the set of norms in $Y$ that are equivalent to $\|\cdot\|$ will be denoted as $\Omega_{HU}.$

\item Drummond-Svaiter functional (see \cite{7})\\

In this case, the parameter space is the set of $w^*-$compact generators of $K^*.$ For a given $w^*-$compact generator $G,$ the Drummond-Svaiter functional associated to $G$ is now defined as:

\begin{equation}\label{funcDS}
\Psi_{G} (y):= \sigma_G(y)= \max\{\langle y^*,y \rangle:  y^*\in G\}.
\end{equation}
When $G$ goes over the set of all generators of $K^*,$ the class of functionals  obtained will be denoted as $\Omega_{DS}.$

\end{itemize}

\end{Definition}

\begin{Remark}

The functionals introduced in Definition \ref{d-func} can be defined in a more general setting and are very important in many fields of mathematics, especially for deriving optimality conditions in vector optimization. Under the assumption that the objective function of a vector optimization problem is locally Lipschitz, Dutta and Tammer \cite{DutTam} derived Lagrangian necessary conditions on the basis of the limiting subdifferential (Mordukhovich \cite{boris2005} and references therein) and the approximate subdifferential (Ioffe \cite{ioffe86}, \cite{ioffe2}, \cite{ioffe}) using a scalarization by means of the functional (\ref{funcGW}). Furthermore, Ha \cite{Ha10-2} used the functional (\ref{funcHU}) and the approximate subdifferential by Ioffe (see \cite{ioffe86}, \cite{ioffe2}, \cite{ioffe}, \cite{32} ) in order to formulate Lagrange multiplier rules for set-valued optimization problems, compare \cite[Chapter 12]{17}.

\end{Remark}

Under Assumption 1, it can be shown (see \cite{3,6,7} and references therein) that each element of every defined class of scalari\-zations is continuous, sublinear, and satisfies the monotonicity and representability properties introduced in Definition \ref{d-axiom}. It turns out that these properties are inherent of DS-functionals, as our next theorem shows.

\begin{Theorem}\label{all is DS}

Let Assumption 1 be fulfilled and consider any continuous and sublinear functional $\Psi:Y\to \Bbb R$ that fulfills the monotonicity and representability properties. Then, $G:= \partial \Psi(0)$ is a $w^*-$compact generator of $K^*$ such that 

$$\Psi = \Psi_G.$$ 

\end{Theorem}
\begin{proof}
Because of the monotonicity and representability assumption, we obtain that $\Psi$ gives a robust representation of $-K.$ By Remark \ref{slaterequalrobustness}, $\Psi$ satisfies the Slater condition. From the second part of Lemma \ref{dualconesublinear}, we now find that 

$$K^*= -(-K)^*= \operatorname{cone}(\partial \Psi(0)),$$ such that $G$ is really a $w^*-$compact convex generator of $K^*.$  From the sublinearity of $\Psi$ and Proposition \ref{properties of convex} $(iv)$, we now have

$$\Psi(y)= \Psi'(0,y)=  \sigma_{\partial \Psi(0)}(y)=\sigma_G(y),$$ as expected. 
\end{proof}

The following proposition shows that $\Omega_{GW}$ is the subset of functionals in $\Omega_{DS}$  associated to the basis of $K^*.$

\begin{Proposition}

Let Assumption 1 be fulfilled. Then, $r \in \operatorname{int}(K)$ if and only if there exists a basis $G$ of $K^*$ (in the $w^*-$ topology) such that 

$$\Psi_r=\Psi_{G}.$$

\end{Proposition}
\begin{proof}

Let $r \in \operatorname{int}(K).$ In virtue of Theorem \ref{all is DS}, we only need to show that $\partial \Psi_r(0)$ is a basis of $K^*.$ The result is then a consequence of Theorem 2.2 in \cite{19}, where it is proved that

\begin{equation}\label{eq:subdif-tammer}
\partial \Psi_r(0)=\{y^* \in K^*: \langle y^*,r \rangle=1\}.
\end{equation}  Conversely, let $G$ be a basis of $K^*$(in the $w^*-$ topology). Because of Theorem 2.2.12 in \cite{24}, there exists $r \in  Y$ such that $\langle y^*,r \rangle>0$ for every $y^* \in K^*\setminus\{0\}$ and

$$G=\{y^*\in K^*: \langle y^*,r \rangle=1\}.$$ Applying now Lemma 3.21 (c) in \cite{9}, we obtain that $r\in \operatorname{int}(K).$ Now, from the subdifferential formula \eqref{eq:subdif-tammer} and the proof of Theorem \ref{all is DS} we have

$$\Psi_r=\Psi_{\partial \Psi_r(0)}=\Psi_{G},$$ as desired. 
\end{proof}

A very important property of the scalarizing functionals in $\Omega_{GW}$ is that of translativity, see \cite[Theorem 2.3.1]{24}. Recall that a functional $\Psi: Y \to \Bbb R$ is said to satisfy the translation property with respect to $r \in \operatorname{int}(K)$ if 

\begin{equation}\label{eq: translativity property}
\forall\; y \in Y, t\in \Bbb R: \;\;\Psi(y +tr) = \Psi(y) +t.
\end{equation}  Next proposition provides another characterization of the class $\Omega_{GW}$ related to this property.

\begin{Proposition}\label{prop: translativity only is GW}

Let Assumption 1 be fulfilled and consider $\Psi_G \in \Omega_{DS},$  $r\in \operatorname{int}(K).$ Then, $\Psi_G$ satisfies the translation property with respect to $r$ if and only if $\Psi_G = \Psi_r \in \Omega_{GW}.$
\end{Proposition} 
\begin{proof}
By substituting $y= 0$ and $t= \pm 1$  in \eqref{eq: translativity property}, we get $\sigma_G(r)=1$ and $\sigma_G(-r) = -1$ respectively. The definition of the support function now implies 

$$\forall\; g^* \in G:\langle g^*,r\rangle \leq 1, \; \; \langle g^*,-r\rangle \leq -1.$$ From this, it can only be

$$\forall\; g^* \in G:\langle g^*,r\rangle=1.$$ Since $G$ is a generator of $K^*,$ we deduce that 

$$G = \{y^* \in K^*: \langle g^*,r\rangle=1\}.$$ Consider now the functional $\Psi_r.$  According to Proposition \ref{properties of convex} $(iii)$ and \eqref{eq:subdif-tammer}, we then obtain 

$$\Psi_r = \sigma_{\partial \Psi_r(0)} = \sigma_G = \Psi_G,$$ as desired.  
\end{proof}

\begin{Remark}
Proposition \ref{prop: translativity only is GW} means that $\Omega_{GW}$ is exactly the set of elements in $\Omega_{DS}$ that satisfy the translation property. It is worth to point out that it was recently established in \cite[Lemma 3.2]{GaoYang} that an element $\Psi_{\|\cdot\|} \in \Omega_{HU}$ satisfies the translation property with respect to $r \in K,$ provided that $d(r,- K) = d(-r, Y \setminus -K) =1.$ However, it turns out that this condition already implies that $\Psi_{\|\cdot\|} = \Psi_r \in \Omega_{GW}.$  Indeed, from the condition we deduce that $r \in \operatorname{int}(K).$ Then, the statement follows from Theorem \ref{all is DS} and Proposition \ref{prop: translativity only is GW}.

\end{Remark}
In the rest of this section, we address the question of the relations between the classes $\Omega_{GW}$ and $\Omega_{HU}.$ The following lemma presents a computation of the Fenchel subdifferential of the HU- functional $\Psi :Y \rightarrow \mathbb{R}$ given by (\ref{funcHU}) at the point $y=0$. For a finite dimensional version of this result, see \cite[Theorem 4.2]{Coulibaly}. For general characterizations (for the case that $K=A$ and $A$ is a subset of $Y$ without convexity assumptions concerning the involved set $A$) of the approximate subdifferential by Ioffe of $\Psi$, see \cite[Proposition 21.11]{Ha10-2}. 

\begin{Lemma}\label{subdif- hu}
Suppose that Assumption 1 is fulfilled. Consider the HU- functional associated to K given by $$\Psi_{\|\cdot \|}(y)= d(y,-K)-d(y,Y\setminus -K).$$ Then, we have:

$$\partial \Psi_{\|\cdot \|} (0)= \overline{\operatorname{conv}}^*\left(K^*\cap S\right),$$ where $$S=\{y^*\in Y^*: \|y^*\|_*=1\},$$ is the unit sphere in the dual space.
\end{Lemma}

\begin{proof}
 Let $\mu: Y \to \mathbb{R}$ be defined as 

$$\mu(y):=  \left\{  \begin{array}{ll}
-d(y, Y\setminus -K), &  \textrm{ if } y\in -K,\\ +\infty, &  \textrm{ if } y\notin -K.
\end{array} \right\}.$$ By Proposition 5 in \cite{3}, we have

\begin{equation}\label{eq:subdifferential of HU decomposed}
\partial \Psi_{\|\cdot \|}(0)= \partial \mu(0)\cap \Bbb B(0,1).
\end{equation} From Proposition 3.1 in \cite{1}, we also have

$$d(y,Y\setminus -K)=\inf_{\|y^*\|_* \geq 1} \left\{\sigma_{(-K)}(y^*)-\langle y^*, y\rangle\right\}.$$ From this, we deduce that

$$-d(y,Y\setminus -K)=\sup_{\|y^*\|_* \geq 1} \left\{\langle y^*, y\rangle-\sigma_{(-K)}(y^*)\right\}.$$  Since $K$ is a cone, it is easy to verify that 

$$\sigma_{(-K)}= \delta_{K^*},$$ the indicator function of $K^*.$ Hence 

$$-d(y,Y\setminus -K)= \sigma_G(y),$$ where $G:= K^*\cap \{y^* \in Y^*:\|y^*\|_* \geq 1\}.$\\
\begin{itemize}
\setlength\itemsep{1em}
\renewcommand{\labelitemi}{\scriptsize$\blacksquare$} 
\item \underline{Claim 1:} For every $y\in -K,$  it holds $-d(y,Y\setminus -K)=\sigma_{\left( K^*\cap S\right)}(y).$\\

Indeed, since $K^*\cap S \subseteq G,$ we obviously have 

$$-d(y,Y\setminus -K)\geq \sigma_{\left( K^*\cap S\right)}(y).$$ Now choose any $y^* \in G$ and $ y\in -K.$ Then, $\langle y^*, y\rangle\leq 0$ and hence 

$$(\|y^*\|_* -1)\langle y^*, y\rangle\leq 0.$$ This implies in particular that 

$$\frac{1}{\|y^*\|_*} y^* \in K^*\cap S \textrm{ and }  \frac{\langle y^*, y\rangle}{\|y^*\|_*}\geq \langle y^*, y\rangle,$$ so our claim is true. 
\end{itemize}
Let 

$$D:= \overline{\operatorname{conv}}^*\left(K^*\cap S\right).$$ Taking into account  \eqref{eq:subdifferential of HU decomposed} and Claim 1 just proved, we have $\bar{y}^* \in \partial \Psi_{\|\cdot \|}(0)$ if and only if $\|\bar{y}^*\|_*\leq 1$ and 

\begin{equation}\label{eq: aux}
\forall\; y\in -K: \langle \bar{y}^*, y \rangle\leq \sigma_{(K^*\cap S)}(y).
\end{equation} By convexity and the $w^*-$ closedness of $\partial \Psi(0)$, it is easy to verify that 

$$D \subseteq \partial \Psi_{\|\cdot \|}(0).$$ In order to finish the proof, we only need to show that the reverse inclusion also holds. 

Assume otherwise. Then there is a 

$$\bar{y}^* \in \partial \Psi_{\|\cdot \|}(0)\setminus D.$$ 

\begin{itemize}
\setlength\itemsep{1em}
\renewcommand{\labelitemi}{\scriptsize$\blacksquare$} 
\item \underline{Claim 2:} $\bar{y}^* \notin [1,+\infty) D.$\\

Since $\bar{y}^* \notin D,$ by Theorem \ref{t - separationtheorem} (ii) we find $\hat{y} \in Y$ such that 

\begin{equation}\label{eq: aux claim [1, inf]D }
\langle \bar{y}^* ,\hat{y} \rangle < \inf_{d^* \in D} \{\langle d^*,\hat{y} \rangle\}.
\end{equation} Taking into account \eqref{eq: aux claim [1, inf]D } and the fact that $\frac{1}{\|\bar{y}^*\|} \bar{y}^* \in D,$ we find that $\langle \bar{y}^* ,\hat{y}\rangle < \frac{1}{\|\bar{y}^*\|} \langle \bar{y}^* ,\hat{y} \rangle .$ Equivalently, we have  

\begin{equation}\label{eq: aux 2 claim [1, inf]D}
(\|\bar{y}^*\|-1)\langle \bar{y}^* ,\hat{y} \rangle<0.
\end{equation} Because of \eqref{eq:subdifferential of HU decomposed} and the fact that $\bar{y}^* \notin D,$ we obtain $\|\bar{y}^*\|_* <1.$ From this and  \eqref{eq: aux 2 claim [1, inf]D}, we get $\langle \bar{y}^* ,\hat{y} \rangle>0.$ Hence, we get

$$0<\langle \bar{y}^* ,\hat{y} \rangle < \inf_{d^* \in D} \{\langle d^*,\hat{y} \rangle\} =  \inf_{d^* \in [1,+\infty) D} \{\langle d^*,\hat{y} \rangle\} .$$ In particular, this implies that $\bar{y}^* \notin [1,+\infty) D,$ and the claim is proved.

\end{itemize} Now, note that $0 \notin D$. Otherwise, we would have $0\in \partial \Psi_{\|\cdot \|}(0) $ and  

$$\forall y \in Y: \quad \langle \bar{y}^*, y \rangle\geq 0 ,$$ which is a contradiction since $K$ is solid and $\langle \bar{y}^*, y \rangle<0$ for each $y\in -\operatorname{int}(K).$ Applying Lemma \ref{wcomp} with $J=[1,+\infty),$ we obtain that the set $[1,+\infty) D$ is $w^*-$ closed (and convex). By Theorem \ref{t - separationtheorem} (ii), we now find  $\bar{y} \in Y$ such that 

$$\langle \bar{y}^*, \bar{y}\rangle> \sigma_{([1,+\infty) D)}  (\bar{y}).$$ This implies that $\langle y^*, \bar{y}\rangle\leq 0$ for each $y^* \in D;$ otherwise, we would have $t y^* \in D$ for every $t\geq 1$ and hence

$$\langle \bar{y}^*, \bar{y}\rangle > \sigma_{([1,+\infty) D)}  (\bar{y})\geq t \langle y^*, \bar{y}\rangle >0. $$ By letting $t \to + \infty$, the right member of this inequality goes to $+\infty$ and we obtain a contradiction. Since $D$ generates $K^*,$ we have $\bar{y}\in -K.$ Hence, 

$$\langle \bar{y}^*, \bar{y}\rangle> \sigma_{([1,+\infty) D)}  (\bar{y}) \geq \sigma_{ D}  (\bar{y})\geq \sigma_{(K^*\cap S)}(\bar{y}),$$ a contradiction to \eqref{eq: aux}. This completes the proof. 
\end{proof}

Now we can establish the hypothesis that guarantee that $\Omega_{GW} \subseteq \Omega_{HU}.$ A result like Theorem \ref{WG es HU} below was first stated in \cite{5}, and later in \cite{4} in the context of set optimization, under a similar argument. 

\begin{Theorem}\label{WG es HU}

Let Assumption 1 be fulfilled, take $r \in \operatorname{int}(K)$ and consider the corresponding element $\Psi_r \in \Omega_{GW},$ i.e,

$$\Psi_r(y)= \inf\{t\in \mathbb{R}: tr\in y+K\}.$$ Furthermore, assume that either $Y$ is reflexive, or that $K$ is normal. Then, there exists  a norm $\|\cdot\|'$ in $Y$ such that:
\begin{itemize}
\setlength\itemsep{1em}
\renewcommand{\labelitemi}{\scriptsize$\blacksquare$} 
\item $\|\cdot\|'$ is equivalent to $\|\cdot\|,$
\item  $\Psi_{\|\cdot\|'}=\Psi_r.$
\end{itemize}

\end{Theorem}
\begin{proof}

The proof will be divided in two cases: one for the reflexivity of $Y,$ and the other one for the normality of $K.$\\

\underline{Case 1:} $Y$ is reflexive.\\
Let

$$V:= \{y^* \in Y^*: |\langle y^*,r\rangle|\leq 1\}$$ and let $B:= \partial \Psi_r(0),$ that is, 

$$B = \{y^*\in Y^*: \langle y^*,r\rangle=1\}.$$ We have that $B$ is $w^*$- compact since it is the subdifferential of the convex and continuous function $\Psi_r$ at $0$. Since $Y$ is a Banach space, as a consequence of the Uniform Boundedness Principle, we must have that $B$ is norm- bounded and so, there exists $M>0$ such that 

$$\|y^*\|_*\leq M \; \textrm{for each }y^* \in B.$$  Define the set 

$$U:= V\cap \Bbb B(0,M).$$ It is easy to see that $U$ is a convex, balanced neighborhood of $0.$ In Figure 1 we illustrate this construction.

Now consider the Minkowski functional associated to $U,$ that is,

$$\rho_U(y^*):=\inf \{t> 0: y^* \in tU\}.$$ By construction, it follows that $\rho_U$ is a norm in $Y^*$ equivalent to $\|\cdot\|_*.$ Let 

$$S:= \{y^* \in Y^*: \rho_U(y^*)=1\}$$ be the unit sphere in $Y^*$ with respect to the norm $\rho_U.$ We claim that 

\begin{equation}\label{eq:identityG*}
B= S\cap K^*.
\end{equation} Indeed, take any $y^* \in B.$ Then, obviously, $y^* \in K^*\cap U$ and hence $\rho_U(y^*)\leq 1.$ If this inequality is strict, then there must be a $t>1$ such that 

$$ty^* \in U,$$ but, then, we would have $$1\geq t\langle y^*,r\rangle=t>1,$$ a contradiction. Hence $\rho_U(y^*)= 1,$ which implies that

$$B\subseteq S\cap K^*.$$ Now take $y^*\in S\cap K^*.$ Since $r\in \operatorname{int}(K),$ we must have $\langle y^*,r\rangle>0.$ Hence 

$$\frac{1}{\langle y^*,r\rangle} y^*\in B \subseteq S.$$ From here we find that 

$$\rho_U \left(\frac{1}{\langle y^*,r\rangle} y^*\right)=1,$$ and so $\langle y^*,r\rangle=1,$ or equivalently, $y^*\in B.$ This proves that 

$$S\cap K^* \subseteq B,$$ and then \eqref{eq:identityG*} is true.\\

Finally, let us define $\|\cdot\|': Y\longrightarrow \mathbb{R}$ by

$$\|y\|'= \sup_{y^*\neq 0^*}\frac{|\langle y^*,y\rangle|}{\rho_U(y^*)}.$$ Then it is easy to check that $\|\cdot\|'$ is equivalent to $\|\cdot\|$ and that $\|\cdot\|'_*=\rho_U$ (this last part is a consequence of reflexivity).  Consider now $\Psi_{\|\cdot\|'}.$ In virtue of Lemma \ref{subdif- hu}, the constructions so far and the fact that $B$ is $w^*-$ closed and convex, we find 

$$\partial\Psi_{\|\cdot\|'}(0)= \overline{\operatorname{conv}}^*(K^*\cap S)= \overline{\operatorname{conv}}^*(B)= B.$$ The result of the theorem follows now from Theorem \ref{all is DS}.\\

\underline{Case 2:} $K$ is normal.\\

Let $I:= (r-K)\cap (-r+K)$ be the order interval associated to $r.$ Because of the normality of $K,$ we have that $I$ is a bounded set (in $\|\cdot\|$) that contains $0$ in its interior, see \cite{30}. Since $K$ is closed and convex, it follows that $I$ is also closed and convex. We consider now the Minkowski functional associated to this set, i.e, 

$$\|y\|'= \inf \{t\geq 0: y\in tI\}.$$ The boundedness of $I$ implies that $\|\cdot\|'$ is a norm in $Y$ equivalent to $\|\cdot\|.$ Next, we claim that 

\begin{equation}\label{eq: gw subd = unit sphere}
S:=\{y^*\in K^*: \|y^*\|_*'=1 \}= \partial \Psi_r(0).
\end{equation} Indeed, take $y^* \in \partial \Psi_r(0).$ From \eqref{eq:subdif-tammer}, we have $y^* \in K^*$ and $\langle y^*,r \rangle=1.$ Now, because of the closedness of $K,$ for any $y \in Y$ with $\|y\|'=1 $ we have $y \in I.$ This means that $r-y\in K,\; y+r\in K.$ Since $y^*\in K^*,$ it follows that 

$$-1= \langle y^*,-r\rangle \leq \langle y^*, y\rangle \leq \langle y^*,r \rangle=1,$$ or equivalently, $|\langle y^*,y \rangle|\leq 1.$ Since $|\langle y^*,r \rangle|= 1,\; \|r\|'=1$ and $y$ was arbitrarily chosen with $\|y\|'=1,$ we obtain $\|y^*\|_*'=1.$ 

Conversely, assume that $y^*\in K^*,$ with $\|y^*\|_*'=1.$ Then, we can find a sequence $\{y_n\}$ in $Y$ such that $\|y_n\|'=1$ for all $n\in \Bbb N$ and $\langle y^*, y_n \rangle \to 1.$ Again, by the closedness of $K,$ we have $\{y_n\}\subseteq  I.$ Since $y^* \in K^*,$ we must have 

$$\langle y^*,y_n \rangle \leq \langle y^*, r \rangle \leq \|y^*\|_*' \|r\|'= 1 $$ and, by letting $n\to +\infty,$ we obtain $\langle y^*, r \rangle=1,$ as desired.

Now, because of \eqref{eq: gw subd = unit sphere} and Lemma \ref{subdif- hu}, we find that 

$$\partial\Psi_{\|\cdot\|'}(0)= \overline{\operatorname{conv}}^*(K^*\cap S)= \overline{\operatorname{conv}}^*(\partial \Psi_r(0))= \partial \Psi_r(0).$$ The result follows. 
\end{proof}

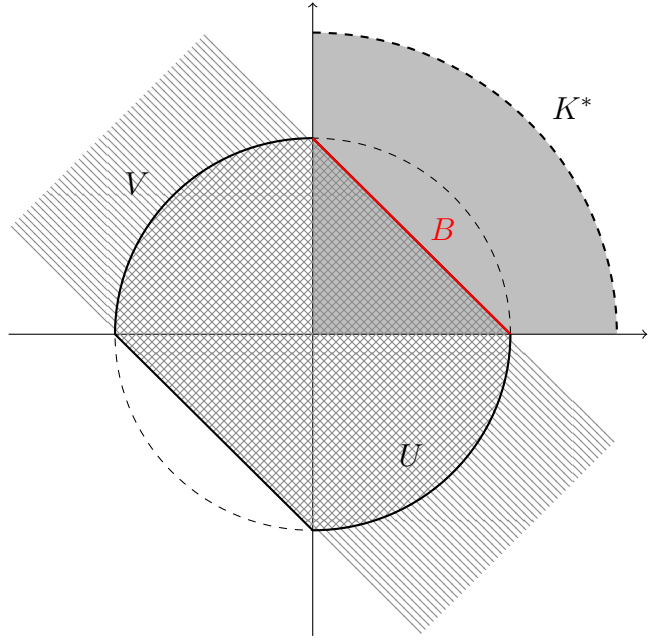
\begin{figure}
\centering
\caption{Geometrical construction in the proof of Theorem \ref{WG es HU}, Case 1}
\begin{tikzpicture}[scale=2]

% cone 
\path [fill= lightgray, dashed] (0,0)-- (90:2) arc (90:0:2)-- (0,0);
\draw[dashed, thick] (90:2) arc (90:0:2);
%\path [fill= lightgray, dashed] (0,0)-- (270:2) arc (180:270:2)-- (0,0);

% axes
\draw[thin,->] (-2,0) -- (2.2,0)node[anchor=north,text width=0.6cm]{} ;
\draw[thin,->] (0,-2) -- (0,2.2)node[anchor=north, text width=0.6cm]{};
\node [right] at (1.5,1.5) {$K^*$};

%basis
\draw[red, thick] (1.3,0) -- (0,1.3);

\node [red, right] at (0.7, 0.7) {$B$};

% define V
\path[ pattern=north west lines, pattern color=black!45] (2,-0.7)-- (-0.7,2)--(-2,0.7)-- (0.7,-2)--cycle;
\node [left] at (-1,1) {$V$};

%boundedness
\draw[dashed] (0,0) circle (1.3cm);

% define U
\path[thick, pattern=north east lines, pattern color=black!45] (1.3,0) arc (0:-90:1.3cm)-- (-1.3,0) arc (180:90:1.3)-- (1.3,0);
\draw[thick,black] (1.3,0) arc (0:-90:1.3cm)-- (-1.3,0) arc (180:90:1.3)-- (1.3,0);
\node [left] at (0.8,-0.8) {$U$};

\draw[red, thick] (1.3,0) -- (0,1.3);
%\path [fill= lightgray, dashed] (0,0)-- (90:2) arc (90:0:2)-- (0,0);

%\draw[pattern=north west lines, pattern color=red] (0,0) -- (2,4)-- (0,1)--(0,0);
%\path [fill= lightgray, dashed] (0,0)-- (90: 2) --  (0 :2)-- (0,0);

%\node [below] at (0,-0.5) {$X^*$};

\end{tikzpicture}
\end{figure}

In Theorem \ref{all is DS} and Theorem \ref{WG es HU} we have shown, under mild assumptions, inclusions between the classes $ \Omega_{GW}, \Omega_{HU}$ and $\Omega_{DS}.$ The following example illustrates that these inclusions are strict in general, i.e, it is possible to have $$\Omega_{GW}\subset \Omega_{HU}\subset  \Omega_{DS}.$$

\begin{Example}\label{ex: inclusions are strict}
Let $Y= \Bbb R^2$ and $ K = \Bbb R^2_+,$ so that $K^* = \Bbb R^2_+.$ Let $\|\cdot\|_2$ and $S$ denote the Euclidean norm and the unit sphere (with respect to $\|\cdot\|_2$) respectively. Next, consider the sets

$$G_1:= \operatorname{conv} (S\cap K^*), \; \; G_2:= G_1 \cup \operatorname{conv}\left(\left\{\begin{pmatrix}
0\\1
\end{pmatrix}, \begin{pmatrix}
1\\0
\end{pmatrix}, \begin{pmatrix}
\frac{1}{2}\\ \frac{1}{2}\end{pmatrix}\right\}\right).$$ Obviously, $G_1$ and $G_2$ are compact generators of $K^*$ and hence $\sigma_{G_1}, \sigma_{G_2}\in \Omega_{DS}.$ 

\underline{Claim:} $\sigma_{G_1} \in \Omega_{HU} \setminus \Omega_{GW}.$

According to Lemma \ref{subdif- hu} and Theorem \ref{all is DS}, it is easy to see that $\sigma_{G_1} \in \Omega_{HU}.$ Assume that $\sigma_{G_1} \in \Omega_{GW}.$ Then, according to the formula for subdifferentials at $0$ of Gerstewitz- Weidner functionals \eqref{eq:subdif-tammer}, we could find an element  $r\in \operatorname{int} (K)$ such that $G_1= \{y^* \in \Bbb R^2_+ : \; \langle y^*, r \rangle =1\}.$ Since the interior of $G_1$ is nonempty, this would be a contradiction.

\underline{Claim:} $\sigma_{G_2} \in \Omega_{DS} \setminus \Omega_{HU}.$

Since we already know that $\sigma_{G_2} \in \Omega_{DS},$ it remains to show that $\sigma_{G_2} \notin \Omega_{HU}.$ Assume otherwise. Then, we can find a norm $\|\cdot\|'$ equivalent to $\|\cdot\|_2$ such that $\sigma_{G_2}= \Psi_{\|\cdot\|'}.$ According to Lemma \ref{subdif- hu}, we now have 

$$G_2 = \partial \Psi_{\|\cdot\|'}(0) = \operatorname{conv}(S' \cap \Bbb R^2_+),$$ were $S'$ is the unit sphere in $\Bbb R^2$ with respect to $\|\cdot\|'_*.$ Let $v:= \begin{pmatrix}
\frac{1}{2}\\ \frac{1}{2}\end{pmatrix}.$ Because $(S \cap \Bbb R^2_+) \cup \{v\}$ are extreme points of $G_2,$ it follows that $(S \cap \Bbb R^2_+) \cup \{v\} \subseteq S' \cap \Bbb R^2_+.$ From this we deduce that $v \in S'$ and $\sqrt{2}v\in S \cap \Bbb R^2_+ \subseteq S',$ a contradiction.
\end{Example}

Taking into account Theorem \ref{all is DS}, Theorem \ref{WG es HU} and Example \ref{ex: inclusions are strict}, we get the following corollary.

\begin{Corollary}

Let Assumption 1 be fulfilled with either $Y$ reflexive, or $K$ normal. Then,

$$\Omega_{GW}\subseteq \Omega_{HU}\subseteq \Omega_{DS}$$ and these inclusions are, in general, strict.

\end{Corollary}

\begin{Remark}
Corollary 2 shows that, since the class $\Omega_{GW}$ is the smallest, its elements only can have additional properties. In particular, according to Proposition \ref{prop: translativity only is GW}, the translation property is one that only the functionals in this class enjoy and that it is exploited in the context of risk measures in mathematical finance, see for example \cite{17}. We conclude that the class $\Omega_{GW} $ is, in this sense, more useful from both the theoretical and practical point of view. 
\end{Remark}

\section{A larger class of scalarizations}\label{s-5}

In this section, we further elaborate on the idea of generators of dual cones to extend the class of scalarizations of Drummond and Svaiter. We will describe a new (and larger) class of scalarizations that are not necessarily convex, but rather quasidifferentiable and positively homogeneous. The following assumption is used through the section:

\begin{Assumption}
Let $(Y,\|\cdot\|)$ be a normed space and $K\subset Y$ a closed, convex and pointed cone.
\end{Assumption}

Let $G,H$ be $w^*-$compact subsets of $Y^*.$ Then, we consider a scalarization functional $\Psi:Y\to \Bbb R$ defined by

\begin{equation}\label{eq: qdfunctional}
\Psi(y):= \sigma_G(y)-\sigma_H(y).
\end{equation} The functionals of the form \eqref{eq: qdfunctional} are quasidifferentiable at any point as a consequence of the quasidifferentiability of the involved support functions. In the following, we study necessary and sufficient geometrical conditions on $G$ and $H$ under which $\Psi$ satisfies the two main axioms of scalarizations: monotonicity and order representability. These conditions will motivate the definition of the new class of quasidifferentiable scalarization functionals. 

First we focus on monotonicity properties based on set relations between the faces of the subdifferential and the superdifferential respectively, see Definition 5. These set relations were introduced by Kuroiwa in \cite{14}  and later by Jahn et. al in \cite{15} in order to compare sets and are very important in set optimization, see \cite{17} and the references therein. For this reason, we start this section by defining them in our context:

\begin{Definition}(See \cite{14,15})
Let Assumption 2 be fulfilled and consider a subset $D$ of $Y^*,$ its dual. Then, $D$ induces binary relations on the subsets of $Y^*$ as follows: for $G,H\subseteq Y^*,$
\begin{itemize}
\setlength\itemsep{1em}
\renewcommand{\labelitemi}{\scriptsize$\blacksquare$} 
\item $H\preceq_l^D G$ ($H$ is lower less than $G$ w.r.t $D$) if $G\subseteq H+D,$
\item $H\preceq_u^D G$ ($H$ is upper less than $G$ w.r.t $D$) if $H\subseteq G-D,$
\item $H\preceq_s^D G$ ($H$ is set less than $G$ w.r.t $D$) if $H\preceq_l^D G$ and $H\preceq_u^D G$
\end{itemize} 
\end{Definition}

Our starting point is the following Lemma, that can be seen as a generalization of H\"ormander's Theorem, see Theorem 2.3.1 in \cite{13}.

\begin{Lemma}\label{jahnlemma}
Let Assumption 2 be fulfilled and let $G,H\subseteq Y^*$ be convex and  $w^*-$ compact. Then, 

\begin{enumerate}
\setlength\itemsep{1em}

\item The equivalence $$H\preceq_u^{K^*} G \Longleftrightarrow \sigma_G|_{K} \geq \sigma_H|_{K}$$ holds.

\item Assume that $\operatorname{int}(K)\neq \emptyset.$ Then,  

$$H\preceq_u^{K^*\setminus\{0\}} G \Longrightarrow \sigma_G|_{\operatorname{int}(K)} > \sigma_H|_{\operatorname{int}(K)}$$ holds. The converse is also true if $H\cap \mathcal{M}(G,K^*)=\emptyset.$ Here, $\mathcal{M}(G, K^*)$ denotes the set of minimal elements of $G$ with respect to the partial order in $Y^*$ induced by $K^*,$ i.e, $\bar{g}^* \in \mathcal{M}(G,K^*)$ if and only if

$$\left(\bar{g}^*-K^*\setminus\{0\}\right)\cap G= \emptyset.$$

\item Assume that $K^{s*}\neq \emptyset.$ Then,

$$H\preceq_u^{K^{s*}} G\Longrightarrow \sigma_G|_{K\setminus\{0\}}> \sigma_H|_{K\setminus\{0\}}.$$ The converse holds if $Y$ is reflexive and $\operatorname{int}(K^*)\neq \emptyset.$

\end{enumerate}
  Here, $f|_A$ denotes the restriction of the functional $f$ to the set $A.$

\end{Lemma}
\begin{proof} 

We only proof (ii) and (iii) since (i) is a particular case of \cite[Lemma 2.1]{27}.\newline

(ii) Assume that $H\subseteq G-K^*\setminus\{0\}.$ Take $y\in \operatorname{int}(K)$ and $h^*\in H$ such that $\langle h^*,y\rangle =\sigma_H(y).$ Then, we can find $g^* \in G: g^* - h^* \in K^*\setminus\{0\}. $ It follows that

$$\sigma_G(y)\geq \langle g^*,y\rangle =\langle g^*-h^*,y\rangle +\langle h^*,y\rangle= \sigma_H(y)+ \langle g^*-h^*,y\rangle > \sigma_H(y) $$ since  $y\in \operatorname{int}(K)$ and hence $ \langle g^*-h^*,y\rangle >0.$ 

In order to prove the second part, assume now that $H\cap \mathcal{M}(G,K^*)=\emptyset.$ If $ \sigma_G|_{\operatorname{int}(K)} > \sigma_H|_{\operatorname{int}(K)},$ by the convexity of $K$ we have $\overline{\operatorname{int}(K)}=K$ and hence in the limit  

$$\sigma_G|_{K} \geq \sigma_H|_{K}.$$ By (i), it follows that $H\subseteq G- K^*.$ Assume that, on the contrary, $H \nsubseteq G- K^*\setminus\{0\}.$ Then, there exists an element $$h^* \in H\cap\left((G-K^*)\setminus (G-K^*\setminus\{0\})\right).$$ From this, we deduce that $h^*\in G\cap H,$ and that $h^* \notin y^* -K^*\setminus\{0\}$ for any $y^*\in G.$ But, by definition, this means that  $h^* \in H\cap \mathcal{M}(G,K^*),$ a contradiction.\newline

(iii)  To this end,  let $K^{*s}\neq \emptyset$ and assume that $H\subseteq G- K^{*s}.$ Take $y \in K\setminus\{0\}$ and $h^* \in H: \langle h^*,y\rangle=\sigma_H(y).$ Analogous to (ii), there exists $g^* \in G$ such that $k^*:= g^* -h^* \in K^{*s}.$ It follows that 
$$\sigma_H(y)= \langle h^*,y \rangle= \langle g^*, y\rangle - \langle k^*,y\rangle \leq \sigma_G(y)- \langle k^*,y\rangle < \sigma_G(y)$$ since $k^* \in K^{*s}.$ Hence the first implication must be true.

Assume now that  $Y$ is reflexive and $\operatorname{int}(K^*)\neq \emptyset.$ It is well known that we always have $\operatorname{int}(K^*)= \operatorname{int}(K^{*s}).$ Furthermore, the reflexivity imply that 

$$\operatorname{int}(K^*)= \operatorname{int}(K^{*s}) = K^{*s}.$$  In fact, this is a characterization of reflexive  spaces, as shown in \cite[Theorem 3.6]{28}.

If $$H\nsubseteq G- K^{*s}= G- \operatorname{int}(K^*),$$ then there must be an $h^*\in H, \; h^* \notin G-\operatorname{int}(K^*).$ By Theorem \ref{t - separationtheorem} (ii), we can find $y \in Y\setminus \{0\}$ such that for all $g^* \in G, k^* \in K^*,$ the inequality

$$\langle h^*,y \rangle \geq \langle g^*-k^*,y\rangle$$ holds. From this it is easy to deduce that $y \in K$ and that $\sigma_H(y)\geq \sigma_G(y),$ as desired. 
\end{proof}

For the forthcoming results, we need the notion of a $y-$ face of a set $A\subseteq Y^*, $ where $y\in Y.$ They are defined as

\begin{equation}\label{y-face}
A^y:=\{y^*\in A : \langle y^*,y \rangle =\sigma_A(y)\}.
\end{equation} By the definition of the set relations, we always have $H\preceq_s^D G \Longrightarrow H\preceq_l^D G $ and $H\preceq_s^D G \Longrightarrow H\preceq_u^D G.$ However, the converse implications are not necessarily true. The following lemma shows that they are equivalent in a specific context.

\begin{Lemma}
Let Assumption 2 be fulfilled and let $G,H\subseteq Y^*$ be convex and  $w^*-$ compact. Then, the following conditions are equivalent:

\begin{enumerate}\label{equivalenceofsetrelations}
\setlength\itemsep{1em}

\item $\forall\; y\in Y: \; H^y\preceq_s^{K^*} G^y, $
\item $\forall\; y\in Y: \; H^y\preceq_u^{K^*} G^y, $
\item $\forall\; y\in Y: \; H^y\preceq_l^{K^*} G^y. $
\end{enumerate}

\end{Lemma}
\begin{proof}
Obviously, $(i)\Longrightarrow (ii)$ and $(i)\Longrightarrow (iii).$ Note that in order to show all of the equivalences, it suffices to show that $(ii)$ and $(iii)$ are equivalents. We now proceed to prove this assertion. The idea lies on the following claim:\\

\begin{itemize}
\setlength\itemsep{1em}
\renewcommand{\labelitemi}{\scriptsize$\blacksquare$} 
\item \underline{Claim:} $\forall\; y\in Y: H^y\preceq_u^{K^*} G^y \Longrightarrow H\preceq_l^{K^*} G.$

Indeed, assume otherwise. Then, there exists $\bar{g}^* \in G \setminus (H +K^*).$ Consider the sets 

$$S= G\cap (H+K^*),\; M= (\bar{g}^*- K^*)\cap G.$$ It is easy to see that  $M$ is $w^*-$ compact. Furthermore, the definition of $\bar{g}^*$ also implies that $M\cap (H+K^*)=\emptyset.$ We can now strongly separate the sets $M$ and $H+K^*$ and obtain an element $\bar{y} \in Y$ such that 
\begin{equation}\label{eq: kmonot1}
\sigma_G(\bar{y})\geq \inf_{g^* \in M}\langle g^*,\bar{y} \rangle > \sup_{y^* \in H+K^*}\langle y^*,\bar{y} \rangle\geq \sigma_S(\bar{y}).
\end{equation}

Now consider the set $G^{\bar{y}}.$ By \eqref{eq: kmonot1}, we must have $G^{\bar{y}}\subseteq G\setminus S$ and this is equivalent to  $G^{\bar{y}}\cap (H+K^*)=\emptyset.$ But this means in particular that $H^{\bar{y}}\nsubseteq G^{\bar{y}}-K^*,$ a contradiction. So, our claim is true. 
\end{itemize}

Assume now that $(ii)$ holds and that there exists $y\in Y$ such that $H^y\npreceq_l^{K^*} G^y,$ or equivalently, that  $G^y\nsubseteq H^y+K^*.$ Then, we can find $\bar{g}^* \in G^y\setminus (H^y+K^*).$ By Theorem \ref{t - separationtheorem} (ii), there exists $\bar{y}\in Y$ such that  

$$\langle \bar{g}^*,\bar{y}\rangle > \sup_{h^*\in H, k^*\in K^*}\langle h^*+k^*,\bar{y}\rangle.$$ In particular, this implies that $\langle k^*,\bar{y} \rangle \leq 0$ for every $k^*\in K^*,$ which, in virtue of Lemma 3.21 a) in \cite{9}, means that $\bar{y}\in -K.$ Since $\bar{g}^*\in G,$ our claim gives us the existence of $\bar{h}^*\in H$ and $\bar{k}^* \in K^*$ such that $\bar{g}^*= \bar{h}^*+\bar{k}^*.$ Then we will have 

$$\sigma_H(\bar{y})\geq \langle \bar{h}^*,\bar{y}\rangle \geq \langle \bar{h}^*,\bar{y}\rangle + \langle \bar{k}^*,\bar{y}\rangle = \langle \bar{g}^*,\bar{y}\rangle > \sup_{h^*\in H, k^*\in K^*}\langle h^*+k^*,\bar{y}\rangle \geq \sigma_H(\bar{y}),  $$ a contradiction. This proves that $(ii)\Longrightarrow (iii).$ By interchanging $G$ and $H$ and considering $(-K^*)$ instead of $K^*,$ a similar analysis proves that $(iii) \Longrightarrow (ii).$ The proof is complete. 
\end{proof}

The following result completely characterizes $K-$ monotone functionals of the form \eqref{eq: qdfunctional} with respect to the corresponding $y$-faces $G^y$ and $H^y$, respectively, given by (\ref{y-face}).

\begin{Lemma}\label{monotonicity}
Let Assumption 2 be fulfilled and let $G,H\subseteq Y^*$ be convex and  $w^*-$ compact. Consider the functional $\Psi:Y \to \Bbb R$ defined by \eqref{eq: qdfunctional}. Then,

\begin{enumerate}
\setlength\itemsep{1em}

\item The functional $\Psi$ is $K-$ monotone if and only if, for all $y\in Y$, the inequality 

$$H^y\preceq_s^{K^*} G^y$$  holds.

\item If $\operatorname{int}(K)\neq \emptyset$ and $H^y\preceq_u^{K^*\setminus\{0\}} G^y $ for all $y \in Y,$ the functional $\Psi$ is strictly $K-$monotone. 

\item If $K^{*s}\neq \emptyset$ and $H^y\preceq_u^{K^{s*}} G^y$ for all $y \in Y,$  the functional $\Psi$ is strongly $K-$monotone.

\end{enumerate}
\end{Lemma}
\begin{proof}
(i) Suppose that $\Psi$ is $K-$ monotone. Note that this is true if and only if for every $y\in Y, \; z \in K,$ the function 

$$\gamma_{y,z}(t):= \Psi(y+tz)$$ attains its minimum over $\Bbb R_+$ at $\bar{t}=0.$ The classical first order necessary optimality condition now implies 
$$\forall\; t\geq 0: \;\gamma'_{y,z}(0,t) \geq 0.$$ It is easy to check that for $t\geq 0,$ we have

$$\gamma'_{y,z}(0,t)= \sigma_{G^y}(tz)-\sigma_{H^y}(tz)=t\left(\sigma_{G^y}(z)-\sigma_{H^y}(z)\right).$$ Hence, the optimality condition is satisfied if and only if 
$$\forall \; y\in Y, \; z \in K:\;\sigma_{G^y}(z)-\sigma_{H^y}(z)\geq 0.$$ Applying Lemma \ref{jahnlemma}(i), we obtain 

$$H^y\subseteq G^y- K^*$$ for every $y \in Y.$ The necessity follows then from Lemma \ref{equivalenceofsetrelations}.

In order to prove sufficiency, assume that $H^y\preceq_s^{K^*} G^y$ for every $y\in Y$ and note that under this condition we have $\gamma'_{y,z}(0,t)\geq 0$ for all $t\geq 0, \;y\in Y,\;z\in K.$ This means that $\gamma_{y,z}$ is increasing along any $z-$ ray, with $z\in K.$ Obviously this means that $0$ is a global minimum of $\gamma_{y,z}$ and hence the $K-$ monotonicity follows.\newline

(ii) Assume now that $\operatorname{int}(K)\neq \emptyset$ and $H^y\subseteq G^y- K^*\setminus\{0\}, $ for all $y \in Y.$ By Lemma \ref{jahnlemma} (ii), this implies 

$$\sigma_{G^y}|_{\operatorname{int}(K)} > \sigma_{H^y}|_{\operatorname{int}(K)}.$$ Hence, for any $y\in Y$, $z\in \operatorname{int}(K)$ and $t>0,$ we have $$\gamma_{y,z}'(0,t)>0.$$ But this would mean that $\bar{t}=0$ is a strict local minimum of the problem $$\min_{t\geq 0}\; \gamma_{y,z}(t).$$ Hence $\Psi$ is strictly monotone.

(iii) Assume in addition that $H^y \subseteq G^y - K^{*s}$ for each $y\in Y.$ By Lemma \ref{jahnlemma} (iii), this implies $$\sigma_{G^y}|_{K\setminus \{0\}} > \sigma_{H^y}|_{K\setminus \{0\}}.$$ Then, we have $\gamma_{y,z}'(0,t)>0$ for every $y\in Y, \;z\in K\setminus \{0\},\;t>0.$ Hence in this case $\bar{t}=0$ is a strict minimum of the problem $$\min_{t\geq 0}\; \gamma_{y,z}(t), $$ which gives us the strong monotonicity of $\Psi.$

\end{proof}
 
Now, we focus on the conditions that $G$ and $H$ must fulfill in order to guarantee the order representability axiom. To this aim, we will need the set valued map $P: Y\to 2^{Y^*}$ defined by

$$P(y):=\{y^*\in Y^*: \langle y^*,y\rangle \geq 0\}= \left(\operatorname*{cone}(\{y\})\right)^*.$$ Note that $$\operatorname{int}(P(y))= \left(\operatorname*{cone}(\{y\})\right)^{s*}\neq \emptyset,$$ given that $y\neq 0.$ With this definition, it is immediate how to find a geometrical condition on $G$ and $H$ that is equivalent to the order representability, as we will show in the following lemma.

\begin{Lemma}\label{representability}
Let Assumption 2 be fulfilled and let $G,H\subseteq Y^*$ be convex and  $w^*-$ compact. Consider the functional $\Psi:Y \to \Bbb R$ defined by \eqref{eq: qdfunctional}. Then,

\begin{enumerate}
\setlength\itemsep{1em}

\item The implication 

$$H\subseteq \bigcap_{y\notin -K}\left(G-\operatorname{int}(P(y))\right) \Longrightarrow\{y\in Y: \Psi(y)\leq 0\}\subseteq -K$$ holds. The converse is true if $Y$ is reflexive. 

\item If in addition $\Psi$ is $K-$ monotone, $K$ is solid and $G\cap H=\emptyset,$ 

$$H\subseteq \bigcap_{y\notin -K}\left(G-\operatorname{int}(P(y))\right) \Longrightarrow -\operatorname{int}(K)=\{y\in Y: \Psi(y)< 0\}.$$
\end{enumerate}

\end{Lemma}
\begin{proof} (i)  We have

\begin{eqnarray*}
H\subseteq \bigcap_{y\notin -K}\left(G-\operatorname{int}(P(y))\right) & \Longleftrightarrow & \forall\; y\notin -K: H\preceq_u^{\operatorname{int}(P(y))} G\\
                             & \Longleftrightarrow &  \forall\; y\notin -K: H\preceq_u^{\left(\operatorname*{cone}(\{y\})\right)^{s*}} G\\ 
                             & \Longrightarrow & \forall\; y\notin -K:\; \sigma_G|_{\operatorname*{cone}(\{y\})\setminus\{0\}}>\sigma_H|_{\operatorname*{cone}(\{y\})\setminus\{0\}} \textrm{(Lemma \ref{jahnlemma} (iii))}\\
                             & \Longleftrightarrow & \forall\; y\notin -K:\; \sigma_G(y)> \sigma_H(y) \\       
                             & \Longleftrightarrow & \forall\; y\notin -K:\; \Psi(y)>0\\ 
                             & \Longleftrightarrow & \{y\in Y: \Psi(y)\leq 0\}\subseteq -K,                            
\end{eqnarray*} as we wanted.

Now, assume that $Y$ is reflexive. In order to prove the converse, it suffices to show that the converse of the one-way implication in the previous proof is true. But this is a consequence of the second part of Lemma \ref{jahnlemma} (iii) by noticing that $\operatorname{int}(P(y))\neq \emptyset$ for every $y\in Y.$ This finishes the proof of (i).

(ii) Because of the monotonicity assumption on $\Psi$ and Lemma \ref{monotonicity} (i), we must have in particular $H^y\subseteq G^y - K^*$ for every $y\in Y$. Moreover, because $G\cap H=\emptyset,$ we must actually have $H^y\subseteq G^y-K^*\setminus\{0\}$ for every $y\in Y,$ or equivalently, $$\forall \;y\in Y: H^y\preceq_u^{K^*\setminus\{0\}} G^y.$$ Applying now Lemma \ref{monotonicity} (ii), we obtain the strict monotonicity of $\Psi.$ Strict monotonicity now implies 

$$-\operatorname{int}(K)\subseteq \{y \in Y: \Psi(y)<0\}.$$ Of course, if $\Psi(y)<0,$ the continuity of $\Psi$ implies $y\in -\operatorname{int}(K).$ The proof is complete. 
\end{proof}

The previous lemmata motivates the following definition:
\begin{Definition}\label{d-scalpair}

Let Assumption 2 be fulfilled and let $G,H$ be convex and $w^*-$compact subsets of $Y^*.$ We say that the pair $[G,H]$ is a scalarization pair if:

\begin{enumerate}
\setlength\itemsep{1em}

\item $H^y \preceq_s^{K^*} G^y$ for every $y \in Y,$
\item $H\cap G=\emptyset,$
\item $H\subseteq \bigcap_{y\notin -K}(G- \operatorname{int}(P(y))).$ 
\end{enumerate} The class of all scalarization pairs is denoted by $\Bbb S(K).$ Furthermore, we define the class of quasidifferentiable and positively homogeneous scalarizing functionals as the set $$\Omega_{QD}:=\{\Psi:Y\to \Bbb R: \exists \;[G,H] \in \Bbb S(K) \textrm{ such that }\Psi=\sigma_G-\sigma_H \}.$$
\end{Definition}

Our next theorem is an immediate consequence of the previous Lemmata. It shows that $\Omega_{QD}$ is a class of functionals whose elements fulfills the monotonicity and representability conditions.

\begin{Theorem}\label{qdscalarizationproperties}
Let $\Psi \in \Omega_{QD}.$ Then, $\Psi$ is $K-$monotone and 

$$-K=\{y \in Y: \Psi(y)\leq 0\}.$$  If $\operatorname{int}(K)\neq \emptyset,$ then $\Psi$ is strictly $K-$monotone and 

$$-\operatorname{int}(K)=\{y \in Y: \Psi(y)<0\}.$$
\end{Theorem}

\begin{proof} The monotonicity follows from Assumption (i) of the scalarization pair in Definition \ref{d-scalpair} and Lemma \ref{monotonicity} (i). By Lemma \ref{representability} (i) and the Assumption (iii) of the scalarization pair, we also obtain $$-K=\{y \in Y: \Psi(y)\leq 0\}.$$ The second part of the proof follows from the proof of Lemma \ref{representability} (ii). 
\end{proof}

Theorem \ref{ds and qd} confirms that $DS-$ functionals are contained in the class $\Omega_{QD}.$

\begin{Theorem}\label{ds and qd}
Let Assumption 1 be fulfilled and let $G,H$ be $w^*-$compact convex subsets of $Y^*.$ Then,
\begin{enumerate}

\setlength\itemsep{1em}

\item The set $G$ is a $w^*-$compact generator of $K^* \Longleftrightarrow [G,\{0\}]\in \Bbb S(K).$ In particular, $\Omega_{DS}\subseteq \Omega_{QD}.$

\item Assume that $[G,H] \in \Bbb S(K)$ and let the functional $\Psi$ be defined by (\ref{eq: qdfunctional}). Then, $$\Psi \in \Omega_{DS} \Longleftrightarrow H+ G\ominus H=G.$$ In particular, the set $G\ominus H$ is necessarily a generator of $K^*.$
\end{enumerate}

\end{Theorem}
\begin{proof} (i) Let us assume first that $G$ is a generator of $K^*.$ We now prove that $[G,\{0\}]\in \Bbb S(K).$ Indeed, by definition, we always have $0\notin G,$ or equivalently, $\{0\} \cap G=\emptyset.$ Furthermore, by Lemma \ref{equivalenceofsetrelations}, the condition $\{0\}\preceq_s^{K^*} G^y$ for every $y\in Y$ is equivalent to  $\{0\}\preceq_l^{K^*} G^y$ for every $y\in Y.$ This just means that $G^y\subseteq K^*$ for all $y,$ which is trivially satisfied by the definition of the generator. In order to finish this first part, it remains to show that 

$$0\in \bigcap_{y\notin -K}\left(G- \operatorname{int}(P(y))\right).$$ Assume otherwise. Then, we could find $y\notin -K$ such that $0\notin G-\operatorname{int}(P(y)).$ This is equivalent to $G\subseteq -P(y),$ so that we have 

$$\langle g^*,y\rangle \leq 0\textrm{ for all }g^*\in G.$$ Because $G$ is a generator of $K^*,$ we can apply Lemma 3.21 (a) in \cite{9} to obtain that $y\in -K,$ a contradiction. This proves the first implication.\\

Now, assume that $[G,\{0\}]\in \Bbb S(K).$ By Theorem \ref{qdscalarizationproperties}, the functional $\Psi:=\sigma_G-\sigma_{\{0\}}=\sigma_G$ satisfies both the monotonicity and the representability axiom. Hence, from the proof of Theorem \ref{all is DS}, we get that $G$ is in fact a generator of $K^*.$

(ii) We have $\Psi \in \Omega_{DS}$ if and only if we could find a generator $D$ of $K^*$ such that $\sigma_G - \sigma_H = \sigma_D.$ Adding $\sigma_H$ to both members we get 

$$\sigma_G= \sigma_D+\sigma_H= \sigma_{D+H}.$$ By H\"ormander's Theorem, we get $G= D+H.$ But then, by the definition of $G\ominus H,$ we must in fact have that $G\ominus H= D.$

\end{proof}

By now, we know that $\Omega_{DS}\subseteq \Omega_{QD}.$ However, it is not clear whether these classes are equivalent. We close this section by showing that this inclusion is actually strict under natural assumptions.

\begin{Theorem}\label{QDlargerclasss}
In addition to Assumption 1, suppose that $\operatorname{int}(K^*)\neq \emptyset.$ Then, $$\Omega_{QD}\setminus \Omega_{DS}\neq \emptyset.$$
\end{Theorem}

\begin{proof}
First, let us note that, if $\operatorname{dim}(Y)=1,$ the result is trivial. Indeed, in this case, w.l.o.g we can assume that $Y=\Bbb R$ and $K=\Bbb R_+.$ Then, it is easy to see that the sets $G=[a,b]$ and $H=[c,d]$ form a scalarization pair iff $a>d.$ By choosing them so that $b-d < a-c,$ we ensure that $G\ominus H =\emptyset$ and hence, by Theorem \ref{ds and qd} $(ii),$ the associated functional to this sets will be nonconvex.\\
For the rest of the proof, we assume that $\operatorname{dim}(Y)>1.$ Here, the proof will be divided in several steps:\\
\underline{\underline {Step 1:}} Definition of suitable subsets $G$ and $H$ of $Y^*.$\\

Take any $r \in \operatorname{int}(K)$ and consider the basis 

$$B:=\{y^* \in K^*: \langle y^*,r\rangle=1\}$$ of $K^*.$ 

Take $v^* \in \operatorname{int}(K^*).$ Hence, we can find $\epsilon >0$ such that 

\begin{equation}\label{eq:ballcontaineddualcone}
\|y^*-v^*\|_* \leq \epsilon \Longrightarrow y^*\in K^*.
\end{equation}
 Since $B$ is $w^*-$compact and $Y$ is Banach, it must be norm- bounded and hence

$$M = \sup_{b^*\in B}\{\|b^*\|_*\}<+\infty.  $$ Consider now the point $p^*= \frac{M}{\epsilon}v^*.$ Then, for any $b^* \in B, $ we have 

$$\epsilon = \frac{\epsilon}{M} M\geq  \frac{\epsilon}{M}\|b^*\|_*=  \left\|v^*-\frac{\epsilon}{M}b^*-v^*\right\|_*. $$ By \eqref{eq:ballcontaineddualcone}, we now have $v^*-\frac{\epsilon}{M}b^* \in K^*,$ which is equivalent to $b^* \in p^* -K^*. $ Since $b^*$ was chosen arbitrarily in $B,$ it follows that the constructed point $p^*$ satisfies 

$$B\subseteq p^*-K^*.$$ Consider now the sets

$$G:= (B+K^*)\cap (p^*-K^*),\;\; C:= \{y \in Y: \langle p^*,y\rangle < \sigma_G(y)\}.$$ We claim that $-r\in C,$ and hence $C\neq \emptyset.$ Indeed, assume otherwise. If $G\neq \{p^*\},$ then we can find $g^* \in G\setminus\{p^*\}.$ Then, we have $p^*-g^* \in K*\setminus\{0\}$ and hence $\langle p^*- g^*, -r \rangle<0$ or equivalently, $$\langle p^*,-r \rangle < \langle g^*,-r \rangle \leq \sigma_G(-r),$$ a contradiction. It follows that $G=\{p^*\},$ and hence also $B=\{p^*\}.$ Now, since $\operatorname{int}(K^*)\neq \emptyset$ and $B$ is a basis of $K^*,$ we deduce that $\operatorname{dim}(Y^*)=1,$ and hence $\operatorname{dim}(Y)=1,$ again a contradiction. 

Let us define next $$\tilde{H}:= \overline{\operatorname{conv}}^*\left(\bigcup_{y\in C}G^y \cup B\right).$$ Furthermore, let 

$$\beta = \inf_{b^*\in B}\{\|b^*\|_*\},\; \;\gamma = \sup_{h^*\in \tilde{H}}\{\|h^*\|_*\}.$$ Finally, put 

$$H:= \frac{\beta}{2\gamma}\tilde{H}.$$ In Figure 2, a geometrical idea of our construction can be observed.\\

\begin{figure}\label{QD strict}
\centering

\caption{Idea of the construction in the proof of Theorem \ref{QDlargerclasss}}
\begin{tikzpicture}[scale=2.5]

% cone 
\path [fill= lightgray, dashed] (0,0)-- (90:2) arc (90:0:2)-- (0,0);
\draw[dashed, thick] (90:2) arc (90:0:2);
%\path [fill= lightgray, dashed] (0,0)-- (270:2) arc (180:270:2)-- (0,0);

% G
\path[thick, pattern=north west lines, pattern color=black!45] (1.2,0) -- (1.2,1.2)-- (0,1.2)-- (0, 0.8)--(0.8,0)-- (1.2,0);
\draw[thick, black] (1.2,0) -- (1.2,1.2)-- (0,1.2)-- (0, 0.8)--(0.8,0)-- (1.2,0);
\node [right] at (0.8,0.8) {$G$};
\node[circle, fill=black, scale=0.5] at (1.2,1.2){};
\node[right] at (1.2,1.2){$p^*$};

%Htilde
\path[thick, pattern=north east lines, pattern color=black!45] (1.2,0) -- (0,1.2)-- (0, 0.8)--(0.8,0)-- (1.2,0);
\draw[thick, black] (1.2,0) -- (1.2,1.2)-- (0,1.2)-- (0, 0.8)--(0.8,0)-- (1.2,0);
\draw[black] (1.2,0)--(0,1.2);
\node [ right] at (0.45,0.45) {$\tilde{H}$};

\draw[red, thick] (0.8,0) -- (0,0.8);
\node[red, below] at (0.8,0) {$B$};

% H
\path[thick, pattern=north west lines, pattern color=black!60] (0.5,0) -- (0,0.5)-- (0,0.2)-- (0.2, 0)--(0.5,0);
\path[thick, pattern=north east lines, pattern color=black!45] (0.5,0) -- (0,0.5)-- (0,0.2)-- (0.2, 0)--(0.5,0);
\draw[thick, black] (0.5,0) -- (0,0.5)-- (0,0.2)-- (0.2, 0)--(0.5,0);
\node [ right] at (0.1,0.1) {$H$};

% axes
\draw[thin,->] (-0.5,0) -- (2.2,0)node[anchor=north,text width=0.6cm]{} ;
\draw[thin,->] (0,-0.5) -- (0,2.2)node[anchor=north, text width=0.6cm]{};
\node [right] at (1.5,1.5) {$K^*$};

%\node [below] at (0,-0.5) {$X^*$};

\end{tikzpicture}
\end{figure}
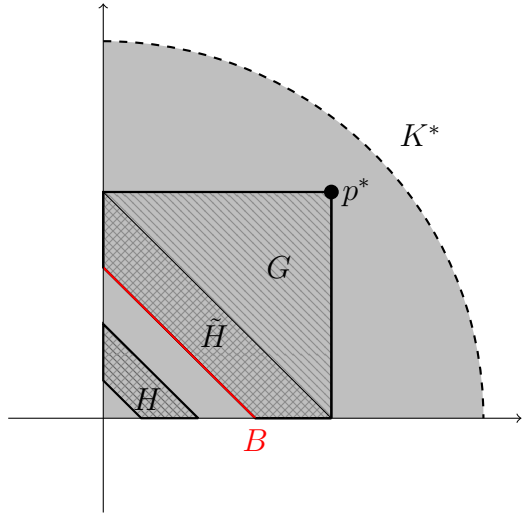

\underline{\underline {Step 2:}} Proving that $[G,H]$ is a scalarization pair.\\

\begin{itemize}
\setlength\itemsep{1em}
\renewcommand{\labelitemi}{\scriptsize$\blacksquare$} 

\item \underline{ $G$ and $H$ are convex and $w^*-$compact}

We have that $G$ is convex because it is the intersection of convex sets. By definition, $\tilde{H}$ is convex, so $H$ is convex too. Since $\tilde{H}$ is a $w^*-$closed subset of $G,$ in order to prove $w^*-$compactness of $H$ it suffices to prove the $w^*-$compactness of $G.$ Moreover, since $Y$ is complete, this is equivalent to show that $G$ is $w^*-$closed and $\|\cdot\|_*$- bounded. 

The $w^*-$closedness of $G$ is easy to see: $G$ is the intersection of two $w^*-$closed sets. In order to see that $G$ is $\|\cdot\|_*$- bounded, note that $\langle g^*,r\rangle \leq \langle p^*,r\rangle$ for every $g^*\in G.$ Now, since $B$ is a basis of $K^*,$ every member $g^* \in G$ can be written as $g^*=\lambda b^*,$ where $\lambda >0$ and $b^*\in B.$ But then, we have the relation $\langle \lambda b^*,r\rangle \leq \langle p^*,r\rangle,$ from which we find that $$\lambda \leq \frac{\langle p^*,r\rangle}{\langle b^*,r\rangle}= \langle p^*,r\rangle.$$ From this, it follows that 

$$\|g^*\|_*= \lambda\|b^*\|_*\leq  \langle p^*,r\rangle \|b^*\|_* \leq M\langle p^*,r\rangle<+\infty,$$ as we wanted.

\item \underline{ $G\cap H=\emptyset.$}

By construction, we have 

$$\forall\; h^* \in H: \|h^*\|_* \leq \frac{\beta}{2\gamma}\gamma=\frac{\beta}{2}< \beta,$$ and also

$$\forall \; g^* \in G: \|g^*\|_*\geq \beta.$$ In particular, this implies that $G\cap H=\emptyset.$

\item \underline{$H^y\preceq_s G^y $ for every $y\in Y.$}

From the proof of Lemma \ref{monotonicity}, we only need to show that 

$$H^y\subseteq G^y-K^*.$$  Before proceding, note that the definition of $\tilde{H}$ implies $$\forall\; y\in C: G^y \subseteq \tilde{H}\subseteq G.$$ Hence, it holds that 

\begin{equation}\label{eq:htildeface}
\forall\; y\in C: \tilde{H}^y =G^y.
\end{equation}

Now, if $y\in C,$ taking into account \eqref{eq:htildeface} and the fact that $\gamma \geq \beta,$ we get 

$$H^y= \frac{\beta}{2\gamma}\tilde{H}^y= \frac{\beta}{2\gamma}G^y \subseteq G^y-K^*.  $$

If on the other hand $y\notin C$ then, by definition, we have $p^*\in G^y.$ But in this case we will have 

$$H^y\subset H\subseteq p^* - K^* \subseteq G^y-K^*,$$ as desired.

\item \underline{$H\subseteq \bigcap_{y\notin -K}(G- \operatorname{int}(P(y))).$}

Assume otherwise. Then,

\begin{eqnarray*}
      \exists \; y\notin -K,\; h^*\in H &:  & \;h^*\notin G- \operatorname{int}(P(y))\\
      & \Longleftrightarrow   &\; h^*\notin G^y- \operatorname{int}(P(y))\\
      & \Longleftrightarrow   & \;\langle h^*,y\rangle \geq \sigma_G(y).\     
\end{eqnarray*}

Since $y\notin -K$ and $G\subseteq K^*,$ we have $\sigma_G(y)>0. $ Then, taking into account that $\frac{2\gamma}{\beta}>1$ and that $\frac{2\gamma}{\beta}h^* \in \tilde{H}\subseteq G,$ we get

$$0< \sigma_G(y)\leq \langle h^*,y\rangle < \langle \underbrace{\frac{2\gamma}{\beta}h^*}_{\in \tilde{H}\subseteq G},y\rangle \leq \sigma_G(y), $$ a contradiction.

\end{itemize}

This proves that $[Q,D]$ is a scalarization pair and hence, by Theorem \ref{qdscalarizationproperties}, it follows that $\Psi := \sigma_G- \sigma_H \in \Omega_{QD}.$\\  
\underline{\underline {Step 3:}} Proving that $\Psi$ is nonconvex\\

Assume that $\Psi$ is convex. By Theorem \ref{ds and qd} (ii), this happens if, and only if, $G \ominus H + H=G.$ In particular, since $p^* \in G,$ it follows that there exists $\bar{y}^*\in Y^*$ such that 

$$p^*\in \bar{y}^* +H,\; \bar{y}^*+H\subseteq G.$$ Let $u^*=p^*-\bar{y}^* \in H .$ In order to arrive at a contradiction, we use the following claims:

\begin{itemize}
\setlength\itemsep{1em}
\renewcommand{\labelitemi}{\scriptsize$\blacksquare$} 
\item \underline{Claim 1:} The inclusion $H\subseteq u^*-K^*$ holds.

Indeeed, take any $h^*\in H.$ Then, by hypothesis, we have $\bar{y}^*+h^* \in G\subseteq p^*- K^*.$ Hence, we can find $k^* \in K^*$ such that $\bar{y}^*+h^*=p^*-k^*.$ But this implies $h^*=p^*-\bar{y}^*-k^*=u^*-k^*\in u^*-K^*.$ Since $h^*$ was chosen arbitrarily in $H,$ this justifies the claim.

\item \underline{Claim 2:} $u^*\notin B'=\frac{\beta}{2\gamma}B.$

Otherwise, note that $B'\subseteq H\subseteq u^*-K^*$ by Claim 1. Take any $b^*\in B'.$ Then, we have $\langle b^*,r\rangle = \langle u^*,r\rangle$ and $u^*-b^*\in K^*.$ But then $$\langle u^*,r\rangle = \langle b^*,r\rangle +\langle u^*-b^*,r\rangle\geq \langle b^*,r\rangle$$ and, since $r\in \operatorname{int}(K),$ the equality holds iff $u^*=b^*.$ Hence $B'=\{u^*\},$ which implies that $\operatorname{dim(K^*)}=1.$ Since $\operatorname{\operatorname{int}(K^*)}\neq \emptyset,$ this in particular means that $\operatorname{dim}(Y^*)=1,$ and hence $\operatorname{dim}(Y)=1,$ a contradiction.

\item \underline{Claim 3:} $u^*\in \operatorname{int}(K^*).$

Assume otherwise. Then, we could apply Theorem \ref{t - separationtheorem} (i) to obtain a functional $y^{**} \in Y^{**}\setminus\{0\}$ such that 

$$\forall \; k^*\in K^*:\langle y^{**},u^* \rangle \leq \langle y^{**}, k^* \rangle.$$ From this we deduce  that $y^{**} \in K^{**}$ and that $\langle y^{**},u^* \rangle\leq 0.$ This, together with Claim 1, the fact that $B'$ is a basis of $K^*$ and that $B'\subseteq H,$ gives us 

$$\forall\; b^*\in B': \langle y^{**},b^* \rangle\leq \langle y^{**},u^*\rangle\leq 0.$$ On the other hand, since $B'$ is in particular a generator of $K^*,$ this implies $\langle y^{**}, k^* \rangle=0$ for any $k^*\in K^*.$ Since $\operatorname{int}(K^*)\neq \emptyset,$ this would imply that $y^{**}=0,$ a contradiction.

\item \underline{Claim 4:} $\left(\bigcup\limits_{y\in C}G^y \cup B\right) \subseteq B\cup \operatorname{bd}(K^*).$

Assume otherwise. Then, we can find $y\in C$ and $y^*\in G^y$ such that $y^*\notin B\cup \operatorname{bd}(K^*).$ Since $G^y\subseteq K^*,$ we must have $y^*\in \operatorname{int}(K^*).$ Since $y\in C,$ we get $\langle y^*,y\rangle > \langle p^*,y\rangle.$ By the definition of $G,$ we have that $k^*:= p^*-y^*\in K^*.$ So, $\langle k^*,y\rangle <0.$ But then, taking into account the fact that $y^*\notin B,$ we get that $y^*-tk^* \in y^*-K^*\subseteq p^*-K^*$ and $y^*-tk^* \in B+K^*$ for $t>0$ small enough. By definition of $G,$ this means that $y^*-tk^* \in G$ for $t>0$ small enough. But, then, we would get 

$$\langle y^*-tk^*,y\rangle =\langle y^*,y\rangle- t\langle k^*,y\rangle > \langle y^*,y\rangle,$$ a contradiction to the fact that $y^*\in G^y.$ The claim is true.

\item \underline{Claim 5:} $\exists\;\alpha>0$ such that $\langle k^*,r\rangle \leq \langle u^*,r \rangle-\alpha$ for every $k^* \in B'\cup \left[(u^*-K^*)\cap \operatorname{bd}(K^*)\right].$

Indeed, because of Claim 2 and 3, we have $u^*\notin B'$ and $u^*\in \operatorname{int}(K^*),$  such that we can find $\delta >0$ such that $\Bbb B(u^*, \delta) \subseteq \operatorname{int}(K^*)$ and $\Bbb B(u^*, \delta)\cap B'=\emptyset.$ The completeness of $Y$ gives us that, for any $\alpha >0,$ the sets

$$B_\alpha= \{y^*\in K^*: \langle y^*,r\rangle=\alpha\}$$ are norm bounded. This implies the existence of $\alpha>0$ such that $B_\alpha\subseteq \Bbb B(0,\delta).$ From Claims 1 and 2 we get that $\frac{\beta}{2\gamma}< \langle u^*,r\rangle$ and, hence, we can choose $\alpha$ small enough such that 

\begin{equation}\label{eq:alphasmall}
\frac{\beta}{2\gamma}\leq \langle u^*,r\rangle-\alpha.
\end{equation} In particular, this means that $\langle k^*,r\rangle \leq \langle u^*,r\rangle-\alpha$ for any $k^*\in B'.$ Then have 
\begin{equation}\label{eq: inclusionintK*}
\{y^* \in u^*-K^* : \langle y^*,r \rangle= \langle u^*,r \rangle -\alpha\}=u^*+B_\alpha \subseteq \Bbb B(u^*, \delta)\subseteq \operatorname{int}(K^*).
\end{equation} Now take any $k^* \in (u^*-K^*)\cap \operatorname{bd}(K^*).$ We have the existence of $\lambda \geq 0$ and $b_\alpha^* \in B_\alpha$ such that $k^*= u^*+\lambda b_\alpha^*.$ In fact, because of \eqref{eq: inclusionintK*}, we must have $\lambda>1.$  Hence, 

$$\langle k^*,r\rangle = \langle u^*,r \rangle +\lambda\langle b_\alpha^*,r\rangle\leq \langle u^*,r \rangle-\alpha,$$ as desired.  
\end{itemize}

Finally, choose any $y\in C.$ By Claim 4, we must have $G^y\subseteq B\cup\operatorname{bd}(K^*).$ Hence we get $\frac{\beta}{2\gamma}G^y \subseteq \left[ B'\cup \operatorname{bd}(K^*)\right]\cap H. $ Because of Claim 1, this implies

\begin{equation}\label{eq:finalinclussion}
\frac{\beta}{2\gamma}G^y \subseteq \left[ B'\cup \operatorname{bd}(K^*)\right]\cap H \subseteq  \left[ B'\cup \operatorname{bd}(K^*)\right]\cap (u^*-K^*)\subseteq B'\cup \left[(u^*-K^*)\cap \operatorname{bd}(K^*)\right].
\end{equation} Now, taking $\alpha$ as in Claim 5, we get $\langle h^*,r\rangle\leq \langle u^*,r \rangle-\alpha $ for any $h^* \in \frac{\beta}{2\gamma}G^y\cup B'.$ Since, by definition, $H=\overline{\operatorname{conv}}^*\left(\bigcup_{y\in C}\frac{\beta}{2\gamma}\left(G^y \cup B\right)\right),$ we must have

$$\forall\; h^*\in H: \langle h^*,r\rangle\leq \langle u^*,r \rangle-\alpha.$$ But this is a contradiction to the fact that $u^*\in H.$ Therefore, the functional $\Psi$ must be nonconvex.

\end{proof}

\section{Conclusions}\label{s-6}

We now briefly summarize the results obtained in this paper:

\begin{itemize}
\setlength\itemsep{1em}
\renewcommand{\labelitemi}{\scriptsize$\blacksquare$} 

\item We found an exact representation for the subdifferential of Hiriart-Urruty functionals. To the best of our knowledge, this representation is new in the infinite dimensional context. Other representations and approximations have been given in the literature, for example in \cite[Proposition 21.11]{Ha10-2}, \cite[Proposition 5]{3} and \cite[Theorem 3]{20}.

\item We provided several relationships in the sense of inclusion between three mayor classes of scalarizations known today, namely that of Gerstewitz ($\Omega_{GW}$), that of Hiriart-Urruty($\Omega_{HU}$) and that of Drummond-Svaiter($\Omega_{DS}$). Our results show that, under natural assumptions, $\Omega_{GW}\subseteq \Omega_{HU} \subseteq \Omega_{DS}.$ Furthermore, we showed that $\Omega_{DS}$ is exactly the set of sublinear scalarizing functionals that satisfy the required axioms to be useful in vector optimization: monotonicity and order representability.

\item We introduced a new class of scalarizing functionals that are not necessarily convex, but instead quasi\-differentiable and positively homogeneous. In order to achieve this, we found geometrical conditions on the quasidifferential of a functional in order to guarantee the fulfillment of the monotonicity and representability axioms. Furthermore, we proved that this class is strictly larger than $\Omega_{DS}$ if the space $Y$ is complete and the dual cone has nonempty interior.
\end{itemize}

The obtained results open new ideas for further research in set optimization, since important connections between set relations and the monotonicity of functionals of the form \eqref{eq: qdfunctional} were shown. The results from Section \ref{s-5} could also be applied in the study of the following problem:

\begin{center}
($\mathcal{P}$)  Consider a functional $f:Y\to \Bbb R$. Then, if it exists, find the largest convex cone $K$ for which $f$ is $K-$monotone. 
\end{center} The solution of this problem could be useful in the development of a benchmark of optimization problems. Indeed, using our ideas automatically generated monotone models can be obtained. It is also of interest to extend the relationships that we have shown in our paper to classes of nonlinear scalarizing functionals related to set valued optimization and problems with variable domination structures.

% Non-BibTeX users please use

\end{document}